\documentclass[11pt]{amsart}

\usepackage{amssymb}
\usepackage{mathrsfs}
\usepackage{amsfonts}
\usepackage{latexsym,amsmath,amsthm,amssymb,amsfonts}
\usepackage[usenames]{color}
\usepackage{xspace,colortbl}
\usepackage{graphicx}
\usepackage{tipa}

\newcommand{\be}{\begin{equation}}
\newcommand{\ee}{\end{equation}}
\newcommand{\beq}{\begin{eqnarray}}
\newcommand{\eeq}{\end{eqnarray}}

\newtheorem{prop}{Proposition}[section]

\newtheorem{remark}[prop]{Remark}

\def\begeq{\begin{equation}}
\def\endeq{\end{equation}}

\def\tr{{\rm tr}}

\def\odot{\setbox0=\hbox{$\bigcirc$}\relax \mathbin {\hbox
to0pt{\raise.5pt\hbox to\wd0{\hfil $\wedge$\hfil}\hss}\box0 }}

\numberwithin{equation} {section}

\numberwithin{equation}{section}
\newtheorem{theorem}{\bf Theorem}[section]
\newtheorem{proposition}[theorem]{\bf Proposition}

\newtheorem{lemma}[theorem]{\bf Lemma}

\newtheorem{corollary}[theorem]{\bf Corollary}

\allowdisplaybreaks

\begin{document}

\title[Inverse Gauss curvature flow in a time cone of Lorentz-Minkowski space]
 {Inverse Gauss curvature flow in a time cone of Lorentz-Minkowski space $\mathbb{R}^{n+1}_{1}$}

\author{
 Ya Gao,\quad Jing Mao$^{\ast}$}

\address{
 Faculty of Mathematics and Statistics, Key Laboratory of
Applied Mathematics of Hubei Province, Hubei University, Wuhan
430062, China. }

\email{Echo-gaoya@outlook.com, jiner120@163.com}

\thanks{$\ast$ Corresponding author}

\date{}
\begin{abstract}
In this paper, we consider the evolution of spacelike graphic
hypersurfaces defined over a convex piece of hyperbolic plane
$\mathscr{H}^{n}(1)$, of center at origin and radius $1$, in the
$(n+1)$-dimensional Lorentz-Minkowski space $\mathbb{R}^{n+1}_{1}$
along the inverse Gauss curvature flow (i.e., the evolving speed
equals the $(-1/n)$-th power of the Gaussian curvature) with the
vanishing Neumann boundary condition, and prove that this flow
exists for all the time. Moreover, we can show that, after suitable
rescaling, the evolving spacelike graphic hypersurfaces converge
smoothly to a piece of the spacelike graph of a positive constant
function defined over the piece of $\mathscr{H}^{n}(1)$ as time
tends to infinity.
\end{abstract}

\maketitle {\it \small{{\bf Keywords}: Inverse Gauss curvature flow,
spacelike hypersurfaces, Lorentz-Minkowski space, Neumann boundary
condition.}

}

\section{Introduction}

The changing shape of a tumbling stone subjected to collisions from
all directions with uniform frequency can be modeled by the motion
of convex surfaces by their Gauss curvature, which was firstly
introduced by Firey \cite{wjf}. Under the assumption of some
existence and regularity of solutions, Firey \cite{wjf} showed that
surfaces which are symmetric about the origin contract to points,
becoming spherical in shape in the evolution process, and he also
conjectured that the result should hold without any symmetry
assumption. 25 years later, this Firey's conjecture was completely
solved by Andrews \cite{ba}. This motion is called \emph{Gauss
curvature flow} (GCF for short) and its importance can be seen from
this well-known history. The study of GCF was intensively carried
out and many other interesting results have been obtained -- see,
e.g., \cite{ba2,ksc,bc,bc2,rsh,lo} and references therein.

Gerhardt \cite{Ge90} (or Urbas \cite{Ur}) firstly considered the
evolution of compact, star-shaped $C^{2,\alpha}$-hypersurfaces
$\mathcal{W}_{0}^{n}$ in the $(n+1)$-dimensional ($n\geq2$)
Euclidean space $\mathbb{R}^{n+1}$ given by
$X_{0}:\mathbb{S}^{n}\rightarrow\mathbb{R}^{n+1}$ along the flow
equation
\begin{eqnarray} \label{ICF-1}
\frac{\partial}{\partial t}X=\frac{1}{F}\nu,
\end{eqnarray}
where $F$ is a positive, symmetric, monotone, homogeneous of degree
one, concave function w.r.t. principal curvatures of the evolving
hypersurfaces
$\mathcal{W}^{n}_{t}=X(\mathbb{S}^{n},t)=X_{t}(\mathbb{S}^{n})$,
\footnote{In this paper, $\mathbb{S}^{n}$ stands for the unit
Euclidean $n$-sphere.}and $\nu$ is the outward unit normal vector of
$\mathcal{W}^{n}_{t}$. They separately proved that the flow exists
for all the time, and, after suitable rescaling, converge
exponentially fast to a uniquely determined sphere of prescribed
radius. This flow is called inverse curvature flow (ICF for short),
and clearly, $F=H$ (the mean curvature) and $F=K^{1/n}$ (the $n$-th
root of the Gaussian curvature) are allowed in this setting, the
flow equation separately become $\partial X/\partial t=\nu/H$ and
$\partial X/\partial t=\nu/K^{1/n}$, which are exactly the IMCF
equation and the inverse Gauss curvature flow (IGCF) equation,
respectively.

The reason why geometers are interested in the study of the theory
of ICFs is that it has important applications in Physics and
Mathematics. For instance, by defining a notion of weak solutions to
IMCF, Huisken and Ilmanen \cite{hi1,hi2} proved the Riemannian
Penrose inequality by using the IMCF approach, which makes an
important step to possibly and completely solve the famous Penrose
conjecture in the General Relativity.  Also using the method of
IMCF, Brendle, Hung and Wang \cite{bhw} proved a sharp Minkowski
inequality for mean convex and star-shaped hypersurfaces in the
$n$-dimensional ($n\geq3$) anti-de Sitter-Schwarzschild manifold,
which generalized the related conclusions in $\mathbb{R}^n$. The
conclusion about the long-time existence and the convergence of IMCF
in the anti-de Sitter-Schwarzschild $n$-manifold ($n\geq3$) obtained
 in \cite{bhw} was
successfully improved to a more general ICF (\ref{ICF-1}) by Chen
and Mao \cite{Ch16}.
 Besides,
applying ICFs, Alexandrov-Fenchel type and other types inequalities
in space forms and even in some warped products can be obtained --
see, e.g., \cite{Gw1,Gw3,Li14,Li16,Mak}. Except \cite{Ch16}, J. Mao
also has some works on ICFs of star-shaped closed hypersurfaces in
Riemannian manifolds (see, e.g., \cite{cmtw,chmw}).

From the above brief introduction, it should be reasonable and
meaningful to study the GCF and its inverse version -- IGCF
(included in ICFs of course).

The examples on ICFs introduced before are only the case that the
initial hypersurface is closed. What about the case that the initial
hypersurfaces have boundary? Can one consider the evolution of
hypersurfaces with boundary along ICFs? The answer is affirmative.
In fact, given a smooth convex open cone in $\mathbb{R}^{n+1}$
($n\geq2$), Marquardt \cite[Theorem 1]{Mar} considered the evolution
of strictly mean convex hypersurfaces with boundary (which are
star-shaped w.r.t. the center of the cone, which meet the cone
perpendicularly and which are contained inside the cone) along the
IMCF, and then, by using the convexity of the cone in the derivation
of the gradient and H\"{o}lder estimates, he proved that this
evolution exists for all the time and the evolving hypersurfaces
converge smoothly to a piece of a round sphere as time tends to
infinity. Inspired by the previous work \cite{cmtw}, Mao and his
collaborator \cite{mt} considered the situation that IMCF equation
in \cite{Mar} was replaced by $\partial X /\partial
t=\nu/|X|^{\beta}H(X)$, $\beta\geq0$ (i.e., the homogeneous
anisotropic factor $|X|^{-\beta}$ was added to the IMCF equation),
and can obtain the long-time existence and a similar asymptotical
behavior of the new flow. This clearly covers Marquardt's result
\cite[Theorem 1]{Mar} as a special case (corresponding to
$\beta=0$). The evolution of strictly convex graphic hypersurfaces
contained in a convex cone in $\mathbb{R}^{n+1}$ ($n\geq2$) along
the IGCF with zero Neumann boundary condition (NBC for short) has
been studied by Sani in 2017, and the long-time existence and the
asymptotical behavior of the flow have been obtained -- see
\cite[Theorems 1.1 and 1.2]{SMG} for details.

In order to state our main conclusion clearly, we need to give
several notions first.

Throughout this paper, let $\mathbb{R}^{n+1}_{1}$ be the
$(n+1)$-dimensional ($n\geq2$) Lorentz-Minkowski space with the
following Lorentzian metric
\begin{eqnarray*}
\langle\cdot,\cdot\rangle_{L}=dx_{1}^{2}+dx_{2}^{2}+\cdots+dx_{n}^{2}-dx_{n+1}^{2}.
\end{eqnarray*}
In fact, $\mathbb{R}^{n+1}_{1}$ is an $(n+1)$-dimensional Lorentz
manifold with index $1$. Denote by
\begin{eqnarray*}
\mathscr{H}^{n}(1)=\{(x_{1},x_{2},\cdots,x_{n+1})\in\mathbb{R}^{n+1}_{1}|x_{1}^{2}+x_{2}^{2}+\cdots+x_{n}^{2}-x_{n+1}^{2}=-1~\mathrm{and}~x_{n+1}>0\},
\end{eqnarray*}
which is exactly the hyperbolic plane of center $(0,0,\ldots,0)$
(i.e., the origin of $\mathbb{R}^{n+1}$) and radius $1$ in
$\mathbb{R}^{n+1}_{1}$. Clearly, from the Euclidean viewpoint,
$\mathscr{H}^{2}(1)$ is one component of a hyperboloid of two
sheets.

As we know, the Lorentz-Minkowski space $\mathbb{R}^{4}_{1}$ is very
important in the study of General Relativity, and it is very
interesting to know whether classical results holding in Euclidean
spaces (or more general, Riemannian manifolds) can be transplanted
to Lorentz-Minkowski spaces (or more general, pesudo-Riemannian
manifolds) or not.\footnote{~One can see the 2$^{nd}$ page of
\cite{GaoY} for a short explanation of the meaningfulness of
considering geometric problems in pesudo-Riemannian manifolds --
using the famous Bernstein theorem in $\mathbb{R}^{n+1}_{1}$ as
example.} Based on this reason, Gao and Mao tried to consider ICFs
in Lorentz-Minkowski spaces and luckily they were successful -- in
fact, the evolution of  (strictly mean convex) spacelike graphic
hypersurface, defined over a convex piece of $\mathscr{H}^{n}(1)$
and contained in a time cone, along the IMCF with zero NBC in
$\mathbb{R}^{n+1}_{1}$ ($n\geq2$) was firstly investigated by them,
and it was shown that the flow exists for all time and, after proper
rescaling, the evolving spacelike graphic hypersurfaces converge
smoothly to a piece of hyperbolic plane of center at origin and
prescribed radius, which actually corresponds to a constant function
defined over the piece of $\mathscr{H}^{n}(1)$, as time tends to
infinity (see \cite[Theorem 1.1]{GaoY2} for details). This result
has already been generalized to its anisotropic version both in
$\mathbb{R}^{n+1}_{1}$ and in the $(n+1)$-dimensional Lorentz
manifold $M^{n}\times\mathbb{R}$, where $M^n$ is a complete
Riemannian $n$-manifold with suitable Ricci curvature constraint
(see \cite{gm1,gm2}). As pointed out in (4) of \cite[Remark
1.1]{gm1}, different from the Euclidean setting made in
\cite[Theorem
 1.1]{mt}, the geometry of
spacelike graphic hypersurfaces in $\mathbb{R}^{n+1}_{1}$ leads to
the fact that
 if one wants to extend the main conclusion in \cite[Theorem
 1.1]{GaoY2} for the IMCF with zero NBC to its anisotropic version in
$\mathbb{R}^{n+1}_{1}$, a totally opposite range for the power of
the homogenous anisotropic factor should be imposed (see
\cite[Theorem 1.1]{gm1}). The lower dimensional version of
\cite{gm1}, i.e., the evolution of spacelike graphic curves defined
over a connected piece of $\mathscr{H}^{2}(1)$ along an anisotropic
IMCF with zero NBC in $\mathbb{R}^{2}_{1}$, has also been considered
and solved completely (see \cite{glm}).

Motivated by Sani's work \cite{SMG} and our previous works
\cite{glm,GaoY2,gm1,gm2}, in this paper, we consider the evolution
of spacelike graphs (contained in a prescribed convex domain) along
the IGCF with zero NBC in $\mathbb{R}^{n+1}_{1}$, and can prove the
following main conclusion.

\begin{theorem}\label{main1.1}
Let $M^n\subset\mathscr{H}^{n}(1)$ be some convex piece of the
hyperbolic plane $\mathscr{H}^{n}(1)\subset\mathbb{R}^{n+1}_{1}$,
and $\Sigma^n:=\{rx\in \mathbb{R}^{n+1}_{1}| r>0, x\in
\partial M^n\}$. Let
$X_{0}:M^{n}\rightarrow\mathbb{R}^{n+1}_{1}$ such that
$M_{0}^{n}:=X_{0}(M^{n})$ is a compact, strictly convex spacelike
$C^{2,\alpha}$-hypersurface ($0<\alpha<1$) which can be written as a
graph over $M^n$.
 Assume that
 \begin{eqnarray*}
M_{0}^{n}=\mathrm{graph}_{M^n}u_{0}
 \end{eqnarray*}
 is a graph over $M^n$ for a positive
map $u_0: M^n\rightarrow \mathbb{R}$, with
$\frac{|Du_{0}|}{u_{0}}\leq\rho<1$ for some nonnegative constant
$\rho$, and
 \begin{eqnarray*}
\partial M_{0}^{n}\subset \Sigma^n, \qquad
\langle\mu\circ X_{0}, \nu_0\circ X_{0} \rangle_{L}|_{\partial
M^n}=0,
 \end{eqnarray*}
 where $\nu_0$ is the past-directed timelike unit normal vector of $M_{0}^{n}$, $\mu$ is a spacelike vector field
defined along $\Sigma^{n}\cap \partial M^{n}=\partial M^{n}$
satisfying the following property:
\begin{itemize}
\item For any $x\in\partial M^{n}$, $\mu(x)\in T_{x}M^{n}$, $\mu(x)\notin T_{x}\partial
M^{n}$, and moreover\footnote{As usual, $T_{x}M^{n}$, $T_{x}\partial
M^{n}$ denote the tangent spaces (at $x$) of $M^{n}$ and $\partial
M^{n}$, respectively. In fact, by the definition of $\Sigma^{n}$
(i.e., a time cone), it is easy to have $\Sigma^{n}\cap\partial
M^{n}=\partial M^{n}$, and we insist on writing as
$\Sigma^{n}\cap\partial M^{n}$ here is just to emphasize the
relation between $\Sigma^{n}$ and $\mu$. Since $\mu$ is a vector
field defined along $\partial M^{n}$, which satisfies $\mu(x)\in
T_{x}M^{n}$, $\mu(x)\notin T_{x}\partial M^{n}$ for any
$x\in\partial M^{n}$, together with the construction of
$\Sigma^{n}$, it is feasible to require $\mu(x)=\mu(rx)$. The
requirement $\mu(x)=\mu(rx)$ makes the assumptions $\langle\mu\circ
X_{0}, \nu_0\circ X_{0} \rangle_{L}|_{\partial M^n}=0$, $\langle \mu
\circ X, \nu\circ X\rangle_{L}=0$ on $\partial M^n \times(0,\infty)$
are reasonable, which can be seen from Lemma \ref{lemma2-1} below in
details. Besides, since $\nu$ is timelike, the vanishing Lorentzian
inner product assumptions on $\mu,\nu$ implies that $\mu$ is
spacelike.}, $\mu(x)=\mu(rx)$.
\end{itemize}
Then we have:

(i) There exists a family of strictly convex spacelike hypersurfaces
$M_{t}^{n}$ given by the unique embedding
\begin{eqnarray*}
X\in C^{2+\alpha,1+\frac{\alpha}{2}} (M^n\times [0,\infty),
\mathbb{R}^{n+1}_{1}) \cap C^{\infty} (M^n\times (0,\infty),
\mathbb{R}^{n+1}_{1})
\end{eqnarray*}
with $X(\partial M^n, t) \subset \Sigma^n$ for $t\geq 0$, satisfying
the following system
\begin{equation}\label{Eq}
\left\{
\begin{aligned}
&\frac{\partial }{\partial t}X=\frac{1}{K^{1/n}}\nu ~~&&in~
M^n \times(0,\infty)\\
&\langle \mu \circ X, \nu\circ X\rangle_{L}=0~~&&on~ \partial M^n \times(0,\infty)\\
&X(\cdot,0)=M_{0}^{n}  ~~&& in~M^{n},
\end{aligned}
\right.
\end{equation}
where $K$ is the Gaussian curvature of
$M_{t}^{n}:=X(M^n,t)=X_{t}(M^{n})$, $\nu$ is the past-directed
timelike unit normal vector of $M_{t}^{n}$.

(ii) The leaves $M_{t}^{n}$ are spacelike graphs over $M^n$, i.e.,
 \begin{eqnarray*}
M_{t}^{n}=\mathrm{graph}_{M^n}u(\cdot, t).
\end{eqnarray*}

(iii) Moreover, the evolving spacelike hypersurfaces converge
smoothly after rescaling to a piece of the spacelike graph of some
positive constant function defined over $M^{n}$, i.e. a piece of
hyperbolic plane of center at origin and prescribed radius.

\end{theorem}

\begin{remark}
\rm{ (1) In fact, $M^{n}$ is some \emph{convex} piece of the
spacelike hypersurface $\mathscr{H}^{n}(1)$ implies that the second
fundamental form of $\partial M^{n}$ is positive definite w.r.t. the
vector field $\mu$ (provided its direction is suitably chosen). \\
 (2) Clearly, the main conclusion in \cite{glm}, i.e. the evolution of spacelike graphic curves defined over a
connected piece of $\mathscr{H}^{2}(1)$ along an anisotropic IMCF
with zero NBC in $\mathbb{R}^{2}_{1}$, can also be seen as the
anisotropic, lower dimensional version of Theorem \ref{main1.1}
here. Based on this reason and our previous experience, it is also
natural to add the homogenous anisotropic factor $|X|^{-\beta}$ to
the RHS of the evolution equation in (\ref{Eq}) and similar
conclusions can be expected under some suitable constraint. We
prefer to leave this as an exercise for readers who are interested
in this topic.}
\end{remark}

This paper is organized as follows. In Section \ref{se2}, we will
recall some useful formulae (such as, the Gauss formula, the
Weingarten formula, several fundamental structure equations, etc) of
spacelike hypersurfaces in $\mathbb{R}^{n+1}_{1}$. In Section
\ref{se3}, we will show that using the spacelike graphic assumption,
the flow equation (which generally is a system of PDEs) changes into
a single scalar second-order parabolic PDE. Several estimates,
including $C^0$, time-derivative and gradient estimates, of
solutions to the flow equation will be shown in Section \ref{se4}.
The most difficult part, $C^{2}$-estimates, will be investigated in
Section \ref{sec5}. This, together with the standard theory of
second-order parabolic PDEs (i.e., Krylov-Safanov theory), can be
used to get the estimates of higher-order derivatives of solutions
to the flow equation and then the long-time existence of the flow
naturally follows. In the end, we will clearly show the convergence
of the rescaled flow in Section \ref{se6}.

\section{The geometry of spacelike hypersurfaces in
$\mathbb{R}^{n+1}_{1}$} \label{se2}

In this section, we prefer to give a brief introduction to some
useful formulae of spacelike graphic hypersurfaces in
$\mathbb{R}^{n+1}_{1}$. One can easily find that the first part of
this section was firstly used in our previous work \cite{GaoY}, and
later almost the whole part was used in \cite{GaoY2}. Readers can
find that the analysis for IGCF shown here is much more complicated
than the one used in \cite{GaoY2} for IMCF, especially in the
$C^2$-estimates part, so for convenience and completeness, we insist
on repeating this content here.

As shown in \cite[Section 2]{GaoY}, we know the following fact:

\textbf{FACT}. Given an $(n+1)$-dimensional Lorentz manifold
$(\overline{M}^{n+1},\overline{g})$, with the metric $\overline{g}$,
and its spacelike hypersurface $M^{n}$. For any $p\in M^{n}$, one
can choose a local  Lorentzian orthonormal frame field
$\{e_{0},e_{1},e_{2},\ldots,e_{n}\}$ around $p$ such that,
restricted to $M^{n}$, $e_{1},e_{2},\ldots,e_{n}$ form orthonormal
frames tangent to $M^{n}$. Taking the dual coframe fields
$\{w_{0},w_{1},w_{2},\ldots,w_{n}\}$ such that the Lorentzian metric
$\overline{g}$ can be written as
$\overline{g}=-w_{0}^{2}+\sum_{i=1}^{n}w_{i}^{2}$. Making the
convention on the range of indices
\begin{eqnarray*}
0\leq I,J,K,\ldots\leq n; \qquad\qquad 1\leq i,j,k\ldots\leq n,
\end{eqnarray*}
and doing differentials to forms $w_{I}$, one can easily get the
following structure equations
\begin{eqnarray}
&&(\mathrm{Gauss~ equation})\qquad \qquad R_{ijkl}=\overline{R}_{ijkl}-(h_{ik}h_{jl}-h_{il}h_{jk}), \label{Gauss}\\
&&(\mathrm{Codazzi~ equation})\qquad \qquad h_{ij,k}-h_{ik,j}=\overline{R}_{0ijk},  \label{Codazzi}\\
&&(\mathrm{Ricci~ identity})\qquad \qquad
h_{ij,kl}-h_{ij,lk}=\sum\limits_{m=1}^{n}h_{mj}R_{mikl}+\sum\limits_{m=1}^{n}h_{im}R_{mjkl},
\label{Ricci}
\end{eqnarray}
and the Laplacian of the second fundamental form $h_{ij}$ of $M^{n}$
as follows
\begin{eqnarray} \label{LF}
&&\Delta
h_{ij}=\sum\limits_{k=1}^{n}\left(h_{kk,ij}+\overline{R}_{0kik,j}+\overline{R}_{0ijk,k}\right)+
\sum\limits_{k=1}^{n}\left(h_{kk}\overline{R}_{0ij0}+h_{ij}\overline{R}_{0k0k}\right)+\nonumber\\
&&\qquad\qquad
\sum\limits_{m,k=1}^{n}\left(h_{mj}\overline{R}_{mkik}+2h_{mk}\overline{R}_{mijk}+h_{mi}\overline{R}_{mkjk}\right)\nonumber\\
&&\qquad\quad
-\sum\limits_{m,k=1}^{n}\left(h_{mi}h_{mj}h_{kk}+h_{km}h_{mj}h_{ik}-h_{km}h_{mk}h_{ij}-h_{mi}h_{mk}h_{kj}\right),
\end{eqnarray}
where $R$ and $\overline{R}$ are the curvature tensors of $M^{n}$
and $\overline{M}^{n+1}$ respectively, $A:=h_{ij}w_{i}w_{j}$ is the
second fundamental form with $h_{ij}$ the coefficient components of
the tensor $A$, $\Delta$ is the Laplacian on the hypersurface
$M^{n}$, and, as usual, the comma ``," in subscript of a given
tensor means doing covariant derivatives. For detailed derivation of
the above formulae, we refer readers to, e.g., \cite[Section
2]{hzl}.

\begin{remark}
\rm{ There is one thing we prefer to mention here, that is, by using
the symmetry of the second fundamental form, we have
$\sum_{m,k=1}^{n}\left(h_{km}h_{mj}h_{ik}-h_{mi}h_{mk}h_{kj}\right)=0$,
which implies that the last term in the RHS of (\ref{LF}) becomes
\begin{eqnarray*}
-\sum\limits_{m,k=1}^{n}\left(h_{mi}h_{mj}h_{kk}-h_{km}h_{mk}h_{ij}\right).
\end{eqnarray*}
Here we insist on writing the Laplacian of $h_{ij}$ as (\ref{LF}) in
order to emphasize the origin of this formula.
 }
\end{remark}

Clearly, in our setting here, all formulae mentioned above can be
used directly with $\overline{M}^{n+1}=\mathbb{R}^{n+1}_{1}$ and
$\overline{g}=\langle\cdot,\cdot\rangle_{L}$.

For convenience, in the sequel we will use the Einstein summation
convention -- repeated superscripts and subscripts should be made
summation from $1$ to $n$. Given an $n$-dimensional Riemannian
manifold $M^{n}$ with the metric $g$, denote by
$\{y^{i}\}_{i=1}^{n}$ the local coordinate of $M^{n}$, and
$\frac{\partial}{\partial y^{i}}$, $i=1,2,\cdots,n$, the
corresponding coordinate vector fields ($\partial_{i}$ for short).
The Riemannian curvature $(1,3)$-tensor $R$ of $M^{n}$ can be
defined by
\begin{equation*}
R(X, Y)Z=-\nabla_{X}\nabla_{Y}Z+\nabla_{Y}\nabla_{X}Z+\nabla_{[X,
Y]}Z,
\end{equation*}
where $X,Y,Z\in\mathscr{X}(M)$ are tangent vector fields in the
tangent bundle $\mathscr{X}(M)$ of $M^{n}$, $\nabla$ is the gradient
operator on $M^{n}$, and, as usual, $[\cdot,\cdot]$ stands for the
Lie bracket. The component of the curvature tensor $R$ is defined by
\begin{equation*}
R\bigg({\frac{\partial}{\partial y^{i}}}, {\frac{\partial}{\partial
y^{j}}}\bigg){\frac{\partial}{\partial y^{k}}} \doteq
R_{ijk}^{l}{\frac{\partial}{\partial y^{l}}}
\end{equation*}
and $R_{ijkl}:= g_{mi}R_{jkl}^{m}$. Now, let us go back to our
setting -- the evolution of strictly convex spacelike graphs in
$\mathbb{R}^{n+1}_{1}$ along the IGCF with zero NBC. The second
fundamental form of the hypersurface $M_{t}^{n}=X(M^{n},t)$ w.r.t.
$\nu$ is given by\footnote{~The definition for the second
fundamental form $h_{ij}$ is different from the one used in
\cite[Section 2]{GaoY2}, although their expressions look like the
same but in fact we have used opposite orientations for the timelike
unit normal vector $\nu$.

One might find that in our works
\cite{glm-1,glm,GaoY,GaoY2,gm1,gm2}, we have used two definitions
for $h_{ij}$, which have exactly opposite sign. However, this would
not create chaos in the analysis of those papers provided other
related settings have been made. In fact, this kind of phenomenons
 happens in the research of Differential Geometry. For instance, one might find that
 there at least exist two definitions for the $(1,3)$-type curvature tensor on
 Riemannian manifolds, which have opposite sign, but \emph{essentially} same
 fundamental equations (such as the Gauss equation, the Codazzi equation, the Ricci identity,
 etc) can be derived provided necessary settings have been made.

  Readers can  check \cite[Remark 1.1]{glm-1} for a very detailed explanation
about the usage of definitions for $h_{ij}$ and other settings made
such that essentially there is no affection to the analysis in our
works mentioned above.
 }
\begin{eqnarray*}
h_{ij}=\left\langle X_{,ij}, \nu\right\rangle_{L},
 \end{eqnarray*}
where $\langle\nu,\nu\rangle_{L}=-1$, $X_{,ij}:=\partial_i
\partial_j X-\Gamma_{ij}^{k}X_k$ with $\Gamma_{ij}^{k}$ the
Christoffel symbols of the metric on $M^{n}_{t}$. Here we would like
to emphasize one thing, that is,
$X_{k}=(X_{t})_{\ast}(\partial_{k})$ with $(X_{t})_{\ast}$ the
tangential mapping induced by the map $X_{t}$. It is easy to have
the following identities
\begin{equation}\label{Gauss for}
X_{,ij}=-h_{ij}\nu, \qquad (\mathrm{Gauss~formula})
\end{equation}
\begin{equation}\label{Wein for}
\nu_{,i}=-h_{ij}X^j, \qquad (\mathrm{Weingarten~formula})
\end{equation}
Besides, using (\ref{Gauss}), (\ref{Codazzi}), (\ref{Ricci}) and
(\ref{LF}) with the fact $\overline{R}=0$ in our setting, we have
\begin{equation}\label{Gauss-1}
R_{ijkl}=h_{il}h_{jk}-h_{ik}h_{jl},
\end{equation}
\begin{equation}\label{Codazzi-1}
\nabla_{k}h_{ij}=\nabla_{j}h_{ik}, \qquad (i.e.,~h_{ij,k}=h_{ik,j})
\end{equation}
and
\begin{eqnarray}\label{Laplace}
\Delta h_{ij}=H_{,ij}-H h_{ik}h^{k}_{j}+h_{ij}|A|^2.
\end{eqnarray}

\begin{remark}
\rm{ Similar to the Riemannian case, the derivation of the formula
(\ref{Laplace}) depends on equations (\ref{Gauss-1}) and
(\ref{Codazzi-1}).}
\end{remark}

We make an agreement that, for simplicity, in the sequel the comma
``," in subscripts will be omitted unless necessary.

\section{The scalar version of the flow equation} \label{se3}

Since the spacelike $C^{2,\alpha}$-hypersurface $M^{n}_{0}$ can be
written as a graph over $M^{n}\subset\mathscr{H}^{n}(1)$, there
exists a function $u_0\in C^{2,\alpha} (M^{n})$ such that
 $X_0: M^{n} \rightarrow \mathbb{R}^{n+1}_{1}$ has the form $x \mapsto
G_{0}:=(x,u_0(x))$. The hypersurface $M_{t}^{n}$ given by the
embedding
\begin{eqnarray*}
X(\cdot, t): M^{n}\rightarrow \mathbb{R}^{n+1}_{1}
\end{eqnarray*}
at time $t$ may be represented as a graph over $M^n\subset
\mathscr{H}^{n}(1)$, and then we can make ansatz
\begin{eqnarray*}
X(x,t)=\left(x,u(x,t)\right)
\end{eqnarray*}
for some function $u: M^{n} \times [0,T) \rightarrow \mathbb{R}$.
The following formulae are needed.

\begin{lemma} \label{lemma2-1}
Define $p:=X(x,t)$ and assume that a point on $\mathscr{H}^{n}(1)$
is described by local coordinates $\xi^{1},\ldots,\xi^{n}$, that is,
$x=x(\xi^{1},\ldots,\xi^{n})$. By the abuse of notations, let
$\partial_i$ be the corresponding coordinate fields on
$\mathscr{H}^{n}(1)$ and
$\sigma_{ij}=g_{\mathscr{H}^{n}(1)}(\partial_i,\partial_j)$ be the
Riemannian metric on $\mathscr{H}^{n}(1)$. Of course,
$\{\sigma_{ij}\}_{i,j=1,2,\ldots,n}$ is also the metric on
$M^{n}\subset\mathscr{H}^{n}(1)$. Following the agreement before,
denote by $u_{i}:=D_{i}u$, $u_{ij}:=D_{j}D_{i}u$, and
$u_{ijk}:=D_{k}D_{j}D_{i}u$ the covariant derivatives of $u$ w.r.t.
the metric $g_{\mathscr{H}^{n}(1)}$, where $D$ is the covariant
connection on $\mathscr{H}^{n}(1)$. Let $\nabla$ be the Levi-Civita
connection of $M_{t}^{n}$ w.r.t. the metric
$g:=u^{2}g_{\mathscr{H}^{n}(1)}-dr^{2}$ induced from the Lorentzian
metric $\langle\cdot,\cdot\rangle_{L}$ of $\mathbb{R}^{n+1}_{1}$.
Then, the following formulae hold:\footnote{~In fact, here we can
treat like what has been done in \cite[Lemma 2.1]{gm1} --
simplifying this whole paragraph as ``\emph{Under the same setting
as \cite[Lemma 3.1]{GaoY2}, we have the following formulas}".
However, for the convenience to readers, we prefer to give the
complete information to the setting of local coordinates and the
induced metric. Besides, one can also easily find that formulas in
\cite[Lemma 3.1]{GaoY2} can be directly used here except the minus
sign should be added to $h_{ij}$ because of the usage of different
definitions.}

(i) The tangential vector on $M_{t}^{n}$ is
\begin{eqnarray*}
X_{i}=\partial_{i}+u_i\partial_{r},
\end{eqnarray*}
and the corresponding past-directed timelike unit normal vector is
given by
\begin{eqnarray*}
\nu=-\frac{1}{v}\left(\partial_r+\frac{1}{u^2}u^j\partial_j\right),
\end{eqnarray*}
where $u^{j}:=\sigma^{ij}u_{i}$, and $v:=\sqrt{1-u^{-2}|D u|^2}$
with $D u$ the gradient of $u$.

(ii) The induced metric $g$ on $M_{t}^{n}$ has the form
\begin{equation*}
g_{ij}=u^2\sigma_{ij}-u_{i} u_{j},
\end{equation*}
and its inverse is given by
\begin{equation*}
g^{ij}=\frac{1}{u^2}\left(\sigma^{ij}+\frac{u^i  u^j
}{u^2v^{2}}\right).
\end{equation*}

(iii) The second fundamental form of $M_{t}^{n}$ is given by
\begin{eqnarray*}
h_{ij}=\frac{1}{v}\left(u_{ij}+u \sigma_{ij}-\frac{2}{u}{u_i  u_j
}\right),
\end{eqnarray*}

(iv) The Gaussian curvature has the form
\begin{eqnarray*}
K=\frac{\det(h_{ij})}{\det(g_{ij})}=\frac{1}{v^{n}}\frac{\det\left(u_{ij}+u
\sigma_{ij}-\frac{2}{u}{u_i  u_j
}\right)}{\det\left(u^2\sigma_{ij}-u_{i} u_{j}\right)}.
\end{eqnarray*}

(v) Let $p=X(x,t)\in \Sigma^n$ with $x\in\partial M^{n}$,
$\hat{\mu}(p)\in T_{p}M^{n}_{t}$, $\hat{\mu}(p)\notin T_{p}\partial
M^{n}_{t}$, $\mu=\mu^i(x_{p})
\partial_{i}(x)$ at $x$, with $\partial _{i}$ the basis vectors of $T_{x}M^{n}$. Then
\begin{eqnarray*}
\langle\hat{\mu}(p), \nu(p) \rangle_{L}=0 \Leftrightarrow \mu^i(x)
u_i(x,t)=0.
\end{eqnarray*}

\end{lemma}

\begin{proof}
The proof is almost the same as that of \cite[Lemma 3.1]{GaoY2}, and
we prefer to omit here.
\end{proof}

Using techniques as in Ecker \cite{Eck} (see also \cite{Ge90, Ge3,
Mar}), the problem \eqref{Eq} is reduced to solve the following
scalar equation with the corresponding initial data and the
corresponding NBC
\begin{equation}\label{Eq-}
\left\{
\begin{aligned}
&\frac{\partial u}{\partial t}=-\frac{v}{K^{1/n}} \qquad
&&~\mathrm{in}~
M^n\times(0,\infty)\\
&\nabla_{\mu} u=0  \qquad&&~\mathrm{on}~ \partial M^n\times(0,\infty)\\
&u(\cdot,0)=u_{0} \qquad &&~\mathrm{in}~M^n.
\end{aligned}
\right.
\end{equation}
By Lemma \ref{lemma2-1}, define a new function $\varphi(x,t)=\log
u(x, t)$ and let $\iota^{ij}(x,t)$ indicate the inverse of
$\iota_{ij}(x,t)=\varphi_{ij}(x,t)+\sigma_{ij}(x)-\varphi_{i}(x,t)\varphi_{j}(x,t)$.
Then the Gaussian curvature can be rewritten as
\begin{eqnarray*}
K=\frac{e^{-n\varphi}}{\left(1-|D\varphi|^{2}\right)^{(n+2)/2}}\frac{\det(\iota_{ij})}{\det(\sigma_{ij})}.
\end{eqnarray*}
Hence, the evolution equation in \eqref{Eq-} can be rewritten as
\begin{eqnarray*}
\frac{\partial}{\partial
t}\varphi=-(1-|D\varphi|^2)^{\frac{n+1}{n}}\frac{\det^{\frac{1}{n}}(\sigma_{ij})}
{\det^{\frac{1}{n}}(\iota_{ij})}:=Q(D\varphi, D^2\varphi).
\end{eqnarray*}

Thus, the problem \eqref{Eq} is again reduced to solve the following
scalar equation with the NBC and the initial data
\begin{equation}\label{Evo-1}
\left\{
\begin{aligned}
&\frac{\partial \varphi}{\partial t}=Q(D\varphi, D^{2}\varphi) \quad
&& \mathrm{in} ~M^n\times(0,T)\\
&\nabla_{\mu} \varphi =0  \quad && \mathrm{on} ~ \partial M^n\times(0,T)\\
&\varphi(\cdot,0)=\varphi_{0} \quad && \mathrm{in} ~ M^n,
\end{aligned}
\right.
\end{equation}
with the matrix
$$\iota_{ij}(x,0)=\varphi_{ij}(x,0)+\sigma_{ij}(x)-\varphi_{i}(x,0)\varphi_{j}(x,0)$$
positive definite up to the boundary $\partial M^n$, since $M_0$ is
strictly convex. Clearly, for the initial spacelike graphic
hypersurface $M_{0}^{n}$,
$$\frac{\partial Q}{\partial \varphi_{ij}}\Big{|}_{\varphi_0}=-\frac{1}{n}\varphi_{t}(x,0)\iota^{ij}(x,0)$$ is positive on $M^n$. Based on the above
facts, as in \cite{Ge90, Ge3, Mar}, we can get the following
short-time existence and uniqueness for the parabolic system
\eqref{Eq}.

\begin{lemma}
Let $X_0(M^n)=M_{0}^{n}$ be as in Theorem \ref{main1.1}. Then there
exist some $T>0$, a unique solution  $u \in
C^{2+\alpha,1+\frac{\alpha}{2}}(M^n\times [0,T]) \cap C^{\infty}(M^n
\times (0,T])$, where $\varphi(x,t)=\log u(x,t)$, to the parabolic
system \eqref{Evo-1} with the matrix
 \begin{eqnarray*}
\iota_{ij}(x,t)=\varphi_{ij}(x,t)+\sigma_{ij}(x)-\varphi_{i}(x,t)\varphi_{j}(x,t)
 \end{eqnarray*}
positive on $M^n$. Thus there exists a unique map $\pi:
M^n\times[0,T]\rightarrow M^n$ such that $\pi(\partial M^n
,t)=\partial M^n$ and the map $\widehat{X}$ defined by
\begin{eqnarray*}
\widehat{X}: M^n\times[0,T)\rightarrow \mathbb{R}^{n+1}_{1}:
(x,t)\mapsto X(\pi(x,t),t)
\end{eqnarray*}
has the same regularity as stated in Theorem \ref{main1.1} and is
the unique solution to the parabolic system \eqref{Eq}.
\end{lemma}

Let $T^{\ast}$ be the maximal time such that there exists some
 \begin{eqnarray*}
u\in C^{2+\alpha,1+\frac{\alpha}{2}}(M^n\times[0,T^{\ast}))\cap
C^{\infty}(M^n\times(0,T^{\ast}))
 \end{eqnarray*}
  which solves \eqref{Evo-1}. In the
sequel, we shall prove a priori estimates for those admissible
solutions on $[0,T]$ where $T<T^{\ast}$.


\section{$C^0$, $\dot{\varphi}$ and gradient estimates} \label{se4}

In order to obtain $C^{0}$-estimate, we prove a comparison principle
by using a standard maximum principle argument.

\begin{lemma}\label{lemma3.1}
Let $\varphi$ and $\psi$ be the two solutions of (\ref{Evo-1}) with
$\varphi(x,0)\leq\psi(x,0)$ for all $x\in M^{n}$, we have
$$\varphi(x,t)\leq\psi(x,t)$$
holds for all $(x,t)\in M^{n}\times[0,T]$.
\end{lemma}

\begin{proof}
Let
$$F(x,t):=\varphi(x,t)-\psi(x,t).$$
It is easy to know
$$F(x,0)\leq 0,$$
since $\varphi$ and $\psi$ are the solution of the scalar equation
(\ref{Evo-1}), so
$$\nabla_{\mu}F=\nabla_{\mu}\varphi-\nabla_{\mu}\psi=0.$$
For a real number $s\in[0,1]$, we set
$$f_{ij}(x,t)[s]:=\sigma_{ij}+s\varphi_{ij}-s\varphi_{i}\varphi_{j}+(1-s)\psi_{ij}-(1-s)\psi_{i}\psi_{j}.$$
Since the set of positive definite matrices is convex, we
have\footnote{~ In (\ref{3.1}), $\det^{\frac{1}{n}}(\sigma_{ij})$
should be $\left(\det(\sigma_{ij})\right)^{1/n}$, and this treatment
also happens to $\det^{\frac{1}{n}}(f_{ij})$. In the sequel, for
convenience and simplicity, sometimes we will directly use
$\det^{\frac{1}{n}}(\cdot)$ to represent the $n$-th root of a
prescribed determinant, i.e. $\left(\det(\cdot)\right)^{1/n}$. }
\begin{equation}\label{3.1}
\begin{split}
\frac{\partial}{\partial t}F(x,t)&=\frac{\partial}{\partial t}\varphi(x,t)-\frac{\partial}{\partial t}\psi(x,t)\\
&= \int_{0}^{1}\frac{d}{ds}\left(
-\frac{\left(1-|D(s\varphi+(1-s)\psi)|^{2}\right)^{\frac{n+1}{n}}\det^{\frac{1}{n}}(\sigma_{ij})}{\det^{\frac{1}{n}}(f_{ij})}\right)ds.
\end{split}
\end{equation}
We do some calculations to get the derivative in the integral of
(\ref{3.1}). Using the formula for the derivative of determinant, we
have
\begin{equation*}
\begin{split}
\frac{d}{ds}{\det}^{\frac{1}{n}}(f_{ij})&=
\frac{1}{n}\left(\det(f_{ij})\right)^{\frac{1}{n}-1}\cdot \det(f_{ij})\cdot f^{ij}\cdot\frac{d}{ds}f_{ij}\\
&=
\frac{1}{n}{\det}^{\frac{1}{n}}(f_{ij})\cdot f^{ij}\cdot\left(\varphi_{ij}-\varphi_{i}\varphi_{j}-\psi_{ij}+\psi_{i}\psi_{j}\right)\\
&= \frac{1}{n}{\det}^{\frac{1}{n}}(f_{ij})\cdot
f^{ij}\cdot\left(F_{ij}-F_{i}(\varphi_{j}+\psi_{j})\right),
\end{split}
\end{equation*}
with $f^{ij}$ the inverse of $f_{ij}$, which is also positive. Then
the symmetry of $f^{ij}$ yields
\begin{equation*}
f^{ij}(\psi_{i}\psi_{j}-\varphi_{i}\varphi_{j})=f^{ij}(\psi_{i}-\varphi_{i})(\varphi_{j}+\psi_{j})=-f^{ij}F_{i}(\varphi_{j}+\psi_{j}).
\end{equation*}
Hence,
\begin{equation*}
\begin{split}
&\frac{d}{ds}\left(1-|D(s\varphi+(1-s)\psi)|^{2}\right)^{\frac{n+1}{n}}\\
&=
-\frac{n+1}{n}\left(1-|D(s\varphi+(1-s)\psi)|^{2}\right)^{\frac{1}{n}}\cdot\frac{d}{ds}|D(s\varphi+(1-s)\psi)|^{2}\\
&=
-2\frac{n+1}{n}\left(1-|D(s\varphi+(1-s)\psi)|^{2}\right)^{\frac{1}{n}}\cdot\sigma^{ij}F_{i}(s\varphi_{j}+(1-s)\psi_{j}).
\end{split}
\end{equation*}

Based on these calculations, the derivative in the integral of
(\ref{3.1}) may be rewritten as
\begin{equation*}
\begin{split}
&\frac{d}{ds}\left(
-\frac{\left(1-|D\left(s\varphi+(1-s)\psi\right)|^{2}\right)^{\frac{n+1}{n}}\det^{\frac{1}{n}}(\sigma_{ij})}{\det^{\frac{1}{n}}(f_{ij})}\right)\\
&~~ =2\frac{n+1}{n}
\frac{(1-|D(s\varphi+(1-s)\psi)|^{2})^{\frac{1}{n}}\sigma^{ij}F_{i}(s\varphi_{j}+(1-s)\psi_{j})\det^{\frac{1}{n}}(\sigma_{ij})}{\det^{\frac{1}{n}}(f_{ij})}\\
&~~ +\frac{1}{n}
\frac{(1-|D(s\varphi+(1-s)\psi)|^{2})^{\frac{n+1}{n}}\det^{\frac{1}{n}}(\sigma_{ij})}{\det^{\frac{1}{n}}(f_{ij})}f^{ij}\left(F_{ij}-F_{i}(\varphi_{j}+\psi_{j})\right).
\end{split}
\end{equation*}
Introduce the following notation for the positive definite
coefficient matrix of the second derivative
\begin{equation*}
a^{ij}(x,t):=\frac{1}{n}{\det}^{\frac{1}{n}}(\sigma_{kl})
\int_{0}^{1}
\frac{\left(1-|D(s\varphi+(1-s)\psi)|^{2}\right)^{\frac{n+1}{n}}f^{ij}}{{\det}^{\frac{1}{n}}(f_{ij})}ds,
\end{equation*}
and set
\begin{equation*}
\begin{split}
b^{i}(x,t):=-a^{ij}(\varphi_{j}+\psi_{j})&+\frac{2(n+1)}{n}{\det}^{\frac{1}{n}}(\sigma_{kl})\cdot\sigma^{ij}\\
&\cdot \int_{0}^{1}
\frac{\left(1-|D(s\varphi+(1-s)\psi)|^{2}\right)^{\frac{1}{n}}(s\varphi_{j}+(1-s)\psi_{j})}{\det^{\frac{1}{n}}(f_{ij})}ds.
\end{split}
\end{equation*}
Then in view of (\ref{3.1}), one has
\begin{equation*}
\left\{
\begin{aligned}
&\frac{\partial}{\partial t}F(x,t)=a^{ij}F_{ij}+b^{k}F_{k} \quad
&& \mathrm{in} ~M^n\times[0,T],\\
&\nabla_{\mu} F =0  \quad && \mathrm{on} ~ \partial M^n\times[0,T],\\
&F(\cdot,0)\leq 0 \quad && \mathrm{in} ~ M^n.
\end{aligned}
\right.
\end{equation*}
Using the parabolic maximum principle and Hopf's Lemma to the above
system, we know that $F$ has to be non-positive for all $t\in
[0,T]$.
\end{proof}

Now, we have:

\begin{lemma}[\bf$C^0$ estimate]\label{$C^{0}$ estimate}
Let $\varphi$ be a solution of \eqref{Evo-1}. Then
\begin{equation*}
c_{1}\leq u(x, t)e^{t} \leq c_{2}, \qquad\quad \forall~ x\in M^n, \
t\in[0,T],
\end{equation*}
where $c_{1}:=\inf\limits_{M^{n}}u(\cdot,0)$,
$c_{2}:=\sup\limits_{M^{n}}u(\cdot,0)$.
\end{lemma}

\begin{proof}
Let $\varphi(x, t)=\varphi(t)$ (independent of $x$) be  the solution
of \eqref{Evo-1} with $\varphi(0)=c$. In this case, the first
equation in \eqref{Evo-1} reduces to an ODE
\begin{eqnarray*}
\frac{d}{d t}\varphi=-1.
\end{eqnarray*}
Therefore,
\begin{eqnarray}\label{blow}
\varphi(t)=-t+c.
\end{eqnarray}
This lemma is an immediate consequence of Lemma \ref{lemma3.1}.
\end{proof}

\begin{lemma}[\bf$\dot{\varphi}$ estimate]\label{lemma3.2}
Let $\varphi$ be a solution of \eqref{Evo-1}. Then we have
\begin{eqnarray*}
\inf_{M^n}\dot{\varphi}(\cdot, 0) \leq \dot{\varphi}(x, t)\leq
\sup_{M^{n}}\dot{\varphi}(\cdot, 0).
\end{eqnarray*}
\end{lemma}
\begin{proof}
Set
\begin{eqnarray*}
\mathcal{M}(x,t)=\dot{\varphi}(x, t).
\end{eqnarray*}
Differentiating both sides of the first evolution equation of
\eqref{Evo-1}, it is easy to get that
\begin{equation*} \label{3.4}
\left\{
\begin{aligned}
&\frac{\partial\mathcal{M}}{\partial t}=
Q^{ij}D_{ij}\mathcal{M}+Q^{k}D_k \mathcal{M} \quad
&& \mathrm{in} ~M^n\times(0,T)\\
&\nabla_{\mu}\mathcal{M}=0 \quad && \mathrm{on} ~\partial M^n\times(0,T)\\
&\mathcal{M}(\cdot,0)=\dot{\varphi}_0 \quad && \mathrm{on} ~ M^n,
\end{aligned}
\right.
\end{equation*}
where $Q^{ij}:=\frac{ \partial Q}{\partial
\varphi_{ij}}=-\frac{1}{n}\varphi_{t}\iota^{ij}$
 and $Q^k:=\frac{ \partial Q}{\partial \varphi_{k}}=-\frac{2\varphi_{t}}{n}\left(\frac{n+1}{1-|D\varphi|^{2}}\sigma^{kl}-\iota^{kl}\right)\varphi_{l}$.
Then the result follows from the maximum principle.
\end{proof}

\begin{lemma}[\bf Gradient estimate]\label{Gradient}
Let $\varphi$ be a solution of \eqref{Evo-1}. Then we have
\begin{equation}\label{Gra-est}
|D\varphi|\leq \sup_{M^n}|D\varphi(\cdot, 0)|\leq \rho<1,
\qquad\quad \forall~ x\in M^n, \ t\in[0,T].
\end{equation}
\end{lemma}
\begin{proof}
Set $\Phi=\frac{|D \varphi|^2}{2}$. By differentiating  $\Phi$, we
have
\begin{equation*}
\begin{aligned}
\frac{\partial \Phi}{\partial t} =\frac{\partial}{\partial
t}\varphi_m \varphi^m = \dot{\varphi}_m\varphi^m =Q_m \varphi^m.
\end{aligned}
\end{equation*}
Then using the evolution equation of $\varphi$ in (\ref{Evo-1})
yields
\begin{eqnarray*}
\frac{\partial \Phi}{\partial t}=Q^{ij}\varphi_{ijm} \varphi^m
+Q^k\varphi_{km} \varphi^m.
\end{eqnarray*}
By the Ricci identity for tensors, we have
\begin{equation*}
\begin{aligned}
\Phi_{ij}&=D_j(\varphi_{mi} \varphi^m)\\&=\varphi_{mij} \varphi^m+\varphi_{mi} \varphi^m_j\\
&=(\varphi_{ijm}+R^l_{imj}\varphi_{l})\varphi^m+\varphi_{mi}\varphi^m_j.
\end{aligned}
\end{equation*}
Therefore, we can express $\varphi_{ijm} \varphi^m$ as
\begin{eqnarray*}
\varphi_{ijm} \varphi^m =\Phi_{ij}-R^l_{imj}\varphi_l
\varphi^m-\varphi_{mi} \varphi^m_j.
\end{eqnarray*}
Then, in view of the fact
$R_{ijml}=\sigma_{il}\sigma_{jm}-\sigma_{im}\sigma_{jl}$ on
$\mathscr{H}^{n}(1)$, we have
\begin{equation}\label{gra}
\begin{aligned}
\frac{\partial \Phi}{\partial t}&=Q^{ij}\Phi_{ij}+Q^k \Phi_k
-Q^{ij}(\varphi_i
\varphi_j-\sigma_{ij}|D\varphi|^2)\\&-Q^{ij}\varphi_{mi}
\varphi^{m}_{j}.
\end{aligned}
\end{equation}
Since the matrix $Q^{ij}$ is positive definite, the third and the
fourth terms in the RHS of \eqref{gra} are non-positive. Since $M^n$
is convex, using a similar argument to the proof of \cite[Lemma
5]{Mar} (see page 1308) implies that
\begin{eqnarray*}
\nabla_{\mu}\Phi=-\sum\limits_{i,j=1}^{n-1}h_{ij}^{\partial
M^{n}}\nabla_{e_i}\varphi\nabla_{e_j}\varphi \leq
0~~~~\qquad\mathrm{on}~\partial M^n\times(0,T),
\end{eqnarray*}
where an orthonormal frame at $x\in\partial M^{n}$, with
$e_{1},\ldots,e_{n-1}\in T_{x}\partial M^{n}$ and $e_{n}:=\mu$, has
been chosen for convenience in the calculation, and
$h_{ij}^{\partial M^{n}}$ is the second fundamental form of the
boundary $\partial M^{n}\subset\Sigma^{n}$.
 So, we can get
\begin{equation*}
\left\{
\begin{aligned}
&\frac{\partial \Phi}{\partial t}\leq Q^{ij}\Phi_{ij}+Q^k\Phi_k
\qquad &&\mathrm{in}~
M^n\times(0,T)\\
&\nabla_{\mu} \Phi \leq 0   && \mathrm{on}~\partial M^n\times(0,T)\\
&\Phi(\cdot,0)=\frac{|D\varphi(\cdot,0)|^2}{2} \qquad
&&\mathrm{in}~M^n.
\end{aligned}\right .\end{equation*}
Using the maximum principle, we have
\begin{equation*}
|D\varphi|\leq \sup_{M^n}|D\varphi(\cdot, 0)|.
\end{equation*}
Since $G_{0}=\{\left(x,u(x,0)\right)|x\in M^{n}\}$ is a spacelike
graph in $\mathbb{R}^{n+1}_{1}$, we have
\begin{equation*}
|D\varphi|\leq \sup_{M^n}|D\varphi(\cdot, 0)|\leq \rho<1,
\qquad\quad \forall~ x\in M^n, \ t\in[0,T].
\end{equation*}
Our proof is finished.
\end{proof}

\begin{remark}
\rm{The gradient estimate in Lemma \ref{Gradient} makes sure that
the evolving graphs $G_{t}:=\{\left(x,u(x,t)\right)|x\in M^{n},0\leq
t\leq T\}$ are spacelike graphs.}
\end{remark}

Combing the gradient estimate (see Lemma \ref{Gradient}) and
$\dot{\varphi}$ estimate (see Lemma \ref{lemma3.2}), we can obtain:
\begin{corollary}
If $\varphi$ satisfies (\ref{Evo-1}), then we have
\begin{equation}\label{det}
0<c_{3}\leq \det(\iota_{ij})\leq c_{4}<+\infty,
\end{equation}
where $c_{3}$ and $c_{4}$ are positive constants independent of
$\varphi$.
\end{corollary}

\section{$C^{2}$ estimates} \label{sec5}

In this section, we will show $C^{2}$ estimates from three aspects
-- interior $C^{2}$-estimates, double normal $C^{2}$-boundary
estimates, and remaining $C^{2}$-boundary estimates. In fact, we can
prove:

\begin{theorem}\label{$C^{2}$ estimate}
Let $\varphi$ be a solution of the flow (\ref{Evo-1}). Then, there
exists $C=C(n, M^{n}_{0})$ such that\footnote{~Clearly,
$C(n,M^{n}_{0})$ stands for a constant depending only on $n$ and
$M^{n}_{0}$. In the sequel, there are many constants appeared, and
we would like to use this way to represent and explain them -- we
mean that we would not count constants by putting numbers in the
subscript except necessary.}
\begin{equation*}
|D^{2}\varphi(x,t)|\leq C(n,M^{n}_{0}) \qquad\qquad \forall(x,t)\in
M^{n}\times[0,T^{\ast}).
\end{equation*}
\end{theorem}

We notice that (\ref{det}), together with the gradient estimate and
$\iota_{ij}>0$, implies a lower bound on $\varphi_{ij}$. Therefore,
we only need to control $\varphi_{ij}$ from above.

\subsection{Interior $C^{2}$-estimates.}
We consider a continuous function
\begin{equation*}
\Theta(\varphi):=\frac{\partial}{\partial t}\varphi-Q(D\varphi,
D^{2}\varphi)
\end{equation*}
for any $\varphi\in C^{2}(M^{n})$ and let
$\dot{\Psi}=\frac{\partial\Psi}{\partial t}$. Then the linearized
operator of $\Theta$ is given by
\begin{equation*}
\begin{aligned}
\mathcal{L}\Psi&:=\frac{d}{ds}|_{s=0}\Theta(\varphi+s\Psi)\\
&= \dot{\Psi}-Q^{ij}\Psi_{ij}-Q^{k}\Psi_{k},
\end{aligned}
\end{equation*}
where
\begin{equation*}
Q^{ij}:=\frac{\partial
Q}{\partial\varphi_{ij}}=-\frac{1}{n}\dot{\varphi}\iota^{ij},
\end{equation*}
and
\begin{equation*}
Q^{k}:=\frac{\partial
Q}{\partial\varphi_{k}}=-\frac{2\dot{\varphi}}{n}\left(\frac{n+1}{1-|D\varphi|^{2}}\sigma^{kl}-\iota^{kl}\right)\varphi_{l}.
\end{equation*}

First, we prove some equalities on $\mathscr{H}^{n}(1)$ which will
play an important role in computations below.

\begin{lemma}\label{two eqaution}
The following equalities hold on $\mathscr{H}^{n}(1)$:
\begin{equation}\label{f1}
\iota^{kl}\iota_{11kl}-\iota^{kl}\iota_{kl11}=-2\tr\iota^{kl}\varphi_{11}-
2\left(\tr\iota^{kl}-n-\iota^{kl}\varphi_{k}\varphi_{l}\right)-2\iota^{kl}\left(\varphi_{1kl}\varphi_{1}-\varphi_{k11}\varphi_{l}\right),
\end{equation}
\begin{equation}\label{f2}
\iota^{kl}\left(\iota_{11k}\iota_{11l}-\iota_{1k1}\iota_{1l1}\right)=
2\iota^{kl}\iota_{11k}\varphi_{l}\iota_{11}-2\iota_{111}\varphi_{1}-(\iota_{11})^{2}\iota^{kl}\varphi_{k}\varphi_{l}+\iota_{11}(\varphi_{1})^{2}.
\end{equation}
\end{lemma}

\begin{proof}
By the Ricci identity for tensors and in view of the fact
$R_{ijml}=\sigma_{il}\sigma_{jm}-\sigma_{im}\sigma_{jl}$ on
$\mathscr{H}^{n}(1)$, we have
\begin{equation*}
\varphi_{11k}=\varphi_{1k1}+\varphi_{k}-\sigma_{1k}\varphi_{1}=\varphi_{k11}+\varphi_{k}-\sigma_{1k}\varphi_{1}.
\end{equation*}
Rewriting it as
\begin{equation}\label{f3}
\iota_{11k}=\varphi_{k11}+\varphi_{k}-\sigma_{1k}\varphi_{1}-2\varphi_{1}\varphi_{1k}.
\end{equation}
Since the covariant derivatives of the curvature tensor for
$\mathscr{H}^{n}(1)$ vanish, by the Ricci identity for tensors, we
have
\begin{equation*}
\begin{aligned}
\varphi_{11kl}&=\left(\varphi_{k11}+R_{11k}^{m}\varphi_{m}\right)_{l}\\
&=
\varphi_{k11l}+R_{11k}^{m}\varphi_{ml}\\
&=
\varphi_{k1l1}+R_{11l}^{m}\varphi_{km}+R_{k1l}^{m}\varphi_{1m}+R_{11k}^{m}\varphi_{ml}\\
&=
\left(\varphi_{kl1}+R_{k1l}^{m}\varphi_{m}\right)_{1}+R_{11l}^{m}\varphi_{km}+R_{k1l}^{m}\varphi_{1m}+R_{11k}^{m}\varphi_{ml}\\
&=
\varphi_{kl11}+R_{11l}^{m}\varphi_{km}+2R_{k1l}^{m}\varphi_{m1}+R_{11k}^{m}\varphi_{ml}.
\end{aligned}
\end{equation*}
It follows that
\begin{equation*}
\begin{aligned}
\iota^{kl}\iota_{11kl}&=
\iota^{kl}\left(\varphi_{11kl}-(\varphi_{1}^{2})_{kl}\right)\\
&=
\iota^{kl}\iota_{kl11}+\iota^{kl}\left(2R_{k1l}^{m}\varphi_{m1}+2R_{11l}^{m}\varphi_{km}-2\varphi_{1kl}\varphi_{1}+2\varphi_{k11}\varphi_{l}\right)\\
&=
\iota^{kl}\iota_{kl11}-2\tr\iota^{kl}\varphi_{11}+2\iota^{kl}\varphi_{kl}-2\iota^{kl}(\varphi_{1kl}\varphi_{1}-\varphi_{k11}\varphi_{l})\\
&=
\iota^{kl}\iota_{kl11}-2\tr\iota^{kl}\varphi_{11}-2\left(\tr\iota^{kl}-n-\iota^{kl}\varphi_{k}\varphi_{l}\right)-2\iota^{kl}\left(\varphi_{1kl}\varphi_{1}-\varphi_{k11}\varphi_{l}\right),
\end{aligned}
\end{equation*}
with $\mathrm{tr}$ the trace operator, which implies the equality
(\ref{f1}).

Now, we pursue the second equality. We can rewrite (\ref{f3}) as
\begin{equation}\label{f4}
\iota_{1k1}=\iota_{11k}-\iota_{11}\varphi_{k}+\iota_{1k}\varphi_{1}.
\end{equation}
Thus,
\begin{equation*}
\begin{aligned}
&\quad
\iota^{kl}(\iota_{11k}\iota_{11l}-\iota_{1k1}\iota_{1l1})\\
&=
\iota^{kl}\iota_{11k}\iota_{11l}\\
&\quad
-\iota^{kl}(\iota_{11k}-\iota_{11}\varphi_{k}+\iota_{1k}\varphi_{1})(\iota_{11l}-\iota_{11}\varphi_{l}+\iota_{1l}\varphi_{1})\\
&=
-2\iota^{kl}\iota_{11k}(-\iota_{11}\varphi_{l}+\iota_{1l}\varphi_{1})\\
&\quad
-\iota^{kl}(-\iota_{11}\varphi_{k}+\iota_{1k}\varphi_{1})(-\iota_{11}\varphi_{l}+\iota_{1l}\varphi_{1})\\
&=
2\iota^{kl}\iota_{11k}\varphi_{l}\iota_{11}-2\iota_{111}\varphi_{1}-(\iota_{11})^{2}\iota^{kl}\varphi_{k}\varphi_{l}+\iota_{11}(\varphi_{1})^{2},
\end{aligned}
\end{equation*}
which finishes the proof of (\ref{f2}).
\end{proof}

\begin{remark}
\rm{Although the equality $\mathrm{(\ref{f2})}$ was also obtained in
\cite{SMG}, the great difference between ours and the one in
\cite{SMG} is rewriting $\mathrm{(\ref{f3})}$ as another form
$\mathrm{(\ref{f4})}$. This improvement simplifies the calculation
in our paper.}
\end{remark}

\begin{lemma}\label{evolution equations}
Under the flow (\ref{Evo-1}), the following evolution equations hold
true

\begin{equation}\label{evolu 1}
\begin{aligned}
\mathcal{L}\left(\frac{1}{2}|D\varphi|^{2}\right)&=\frac{1}{n}\dot{\varphi}
\left((1-|D\varphi|^{2})\iota^{ij}\sigma_{ij}
 -  (1-|D\varphi|^{2})\iota^{ij}\varphi_{i}\varphi_{j} + \Delta\varphi
 + |D\varphi|^{2} - n\right),
\end{aligned}
\end{equation}
\begin{equation}\label{evolu 2}
\begin{aligned}
\mathcal{L}\iota_{11}&=\frac{(\dot{\varphi}_{1})^{2}}{\dot{\varphi}} - \frac{1}{n}\dot{\varphi}\iota^{kl}_{,1}\iota_{kl,1} - \frac{4(n+1)\dot{\varphi}}{n}\frac{1}{(1-|D\varphi|^{2})^{2}}(\sigma^{kl}\varphi_{k}\varphi{l1})^{2}\\
&\quad
-\frac{2(n+1)\dot{\varphi}}{n}\frac{1}{1-|D\varphi|^{2}}\sigma^{kl}\varphi_{k1}
\varphi_{l1} - \frac{2\dot{\varphi}}{n}\frac{(n+1)}{1-|D\varphi|^{2}}\left( (\varphi_{1})^{2} - |D\varphi|^{2} \right)\\
&\quad - \frac{2}{n}\dot{\varphi}(\iota_{11}\tr\iota^{kl} - n),
\end{aligned}
\end{equation}
\begin{equation}\label{evolu 3}
\begin{aligned}
\mathcal{L}(\tau^{l}\varphi_{l})&=-\frac{1}{n}\dot{\varphi}\iota^{ij}
\left( \sigma_{ij}\varphi_{l}\tau^{l} -
\sigma_{il}\varphi_{j}\tau^{l}
 - 2\varphi_{li}\tau^{l}_{,j} - \varphi_{l}\tau^{l}_{,ij}
  + 2\varphi_{j}\varphi_{l}\tau^{l}_{,i}\right)\\
  &\quad
  + \frac{2\dot{\varphi}}{n}\frac{n+1}{1-|D\varphi|^{2}}\varphi^{k}\varphi_{l}
  \tau^{l}_{,k},
\end{aligned}
\end{equation}
where $\Delta$ is the Laplace of $D$ and $\tau^{i}:
\overline{M^{n}}\to \mathbb{R}$ is a smooth function that does not
depend on $\varphi$.

\end{lemma}

\begin{proof}
Clearly,
$$\mathcal{L}(\frac{1}{2}|D\varphi|^{2})=\sigma^{kl}\varphi_{k}\dot{\varphi}_{l}-Q^{ij}\left(\frac{1}{2}|D\varphi|^{2}\right)_{ij}-Q^{k}\left(\frac{1}{2}|D\varphi|^{2}\right)_{k}.$$
Using the evolution equation in (\ref{Evo-1}), the first term in the
RHS of the above equation becomes
$$\sigma^{kl}\varphi_{k}\dot{\varphi}_{l}=\sigma^{kl}(Q^{ij}\varphi_{ijl}+Q^{s}\varphi_{sl})\varphi_{k}.$$
Then we have
\begin{equation*}
\begin{aligned}
\varphi_{ijl}\varphi^{l} &=
(\varphi_{lij} + R^{m}_{ijl}\varphi_{m})\varphi^{l} = \varphi_{lij}\varphi^{l} + \sigma_{ij}\varphi_{l}\varphi^{l} - \sigma_{il}\varphi_{j}\varphi^{l}\\
&= (\frac{1}{2}|D\varphi|^{2})_{ij} -
\sigma^{kl}\varphi_{kj}\varphi_{li} + |D\varphi|^{2}\sigma_{ij} -
\varphi_{i}\varphi_{j}.
\end{aligned}
\end{equation*}
Hence,
$$\mathcal{L}(\frac{1}{2}|D\varphi|^{2}) = -Q^{ij}\varphi_{kj}\varphi_{li}\sigma^{kl} +
Q^{ij}(|D\varphi|^{2}\sigma_{ij}-\varphi_{i}\varphi_{j}).$$ It
follows that
\begin{equation*}
\mathcal{L}(\frac{1}{2}|D\varphi|^{2}) =
\frac{1}{n}\dot{\varphi}\left(
(1-|D\varphi|^{2})\iota^{ij}\sigma_{ij} -
(1-|D\varphi|^{2})\iota^{ij}\varphi_{i}\varphi_{j} + \Delta\varphi +
|D\varphi|^{2} - n \right)
\end{equation*}
in view of
$$\iota^{ij}\varphi_{kj} = \iota^{ij}(\iota_{kj} - \sigma_{kj} + \varphi_{k}\varphi_{j})
= \delta^{i}_{k} - \iota^{ij}\sigma_{kj} +
\iota^{ij}\varphi_{k}\varphi_{j}.$$
 This finishes the proof of (\ref{evolu 1}).

Now, we are going to prove (\ref{evolu 2}). Clearly,
$$\mathcal{L}(\iota_{11}) = \dot{\iota}_{11} - Q^{ij}\iota_{11,ij} - Q^{k}\iota_{11,k}.$$
Using the evolution equation in \eqref{Evo-1}, we have
\begin{equation*}
\begin{aligned}
\dot{\iota}_{11} &= \dot{\varphi}_{11} - 2\dot{\varphi}_{1}\varphi_{1}\\
&= -\left( \frac{\dot{\varphi}}{n}\iota^{kl}\iota_{kl,1} +
\frac{\dot{\varphi}}{n}\frac{2(n+1)}{1-|D\varphi|^{2}}\sigma^{kl}\varphi_{k}\varphi_{l1}
\right)_{1} -
2\dot{\varphi}_{1}\varphi_{1}\\
&= \frac{(\dot{\varphi}_{1})^{2}}{\dot{\varphi}} -
\frac{\dot{\varphi}}{n}\iota^{kl}_{,1}\iota_{kl,1} -
\frac{\dot{\varphi}}{n}\iota^{kl}\iota_{kl,11} -
\frac{4(n+1)\dot{\varphi}}{n}\frac{1}{(1-|D\varphi|^{2})^{2}}(\sigma^{kl}\varphi_{k}\varphi_{l1})^{2}\\
&\quad - \frac{2(n+1)\dot{\varphi}}{n}\frac{1}{1-|D\varphi|^{2}}
\sigma^{kl}\left(\varphi_{k1}\varphi_{l1} +
\varphi_{k}\varphi_{l11}\right) - 2\dot{\varphi}_{1}\varphi_{1}.
\end{aligned}
\end{equation*}
Inserting \eqref{f1} into the above equality results in
\begin{equation*}\label{f8}
\begin{aligned}
\dot{\iota}_{11} &= \frac{(\dot{\varphi}_{1})^{2}}{\dot{\varphi}} -
\frac{\dot{\varphi}}{n}\iota^{kl}_{,1}\iota_{kl,1} -
\frac{\dot{\varphi}}{n}\iota^{kl}\iota_{11,kl} -
\frac{2}{n}\dot{\varphi}\tr\iota^{kl}\varphi_{11}
- \frac{2}{n}\dot{\varphi}\left(\tr\iota^{kl} - n - \iota^{kl}\varphi_{k}\varphi_{l}\right)\\
&\quad -
\frac{4(n+1)\dot{\varphi}}{n}\frac{1}{(1-|D\varphi|^{2})^{2}}(\sigma^{kl}\varphi_{k}\varphi_{l1})^{2}
- \frac{2(n+1)\dot{\varphi}}{n}\frac{1}{1-|D\varphi|^{2}}\sigma^{kl}\varphi_{k1}\varphi_{l1}\\
&\quad -
\frac{2\dot{\varphi}}{n}\left(\frac{(n+1)}{1-|D\varphi|^{2}}\sigma^{kl}
- \iota^{kl}\right)\varphi_{l}\varphi_{k11} -
2\left(\dot{\varphi}_{1} +
\frac{1}{n}\dot{\varphi}\iota^{kl}\varphi_{1kl}\right)\varphi_{1}.
\end{aligned}
\end{equation*}
Since
\begin{equation*}
\begin{aligned}
-2\iota^{kl}\varphi_{1kl}\varphi_{1} &=
-2\iota^{kl}\left( \varphi_{kl1} + \sigma_{1k}\varphi_{l} - \sigma_{kl}\varphi_{1} \right)\varphi_{1}\\
&= - 2\iota^{kl}\iota_{kl,1}\varphi_{1} -
4\iota^{kl}\varphi_{k}\varphi_{l1}\varphi_{1} -
2\iota^{kl}\sigma_{k1}\varphi_{l}\varphi_{1} +
2\iota^{kl}\sigma_{kl}(\varphi_{1})^{2},
\end{aligned}
\end{equation*}
 we have
\begin{equation*}
-2(\dot{\varphi}_{1} +
\frac{1}{n}\dot{\varphi}\iota^{kl}\varphi_{1kl})\varphi_{1} =
\frac{4\dot{\varphi}}{n}\frac{(n+1)}{1-|D\varphi|^{2}}\sigma^{kl}\varphi_{k}\varphi_{l1}\varphi_{1}
-\frac{2\dot{\varphi}}{n}\iota^{kl}\left(
2\varphi_{k}\varphi_{l1}\varphi_{1} +
\sigma_{k1}\varphi_{l}\varphi_{1} - \sigma_{kl}(\varphi_{1})^{2}
\right)
\end{equation*}
and
\begin{equation*}
-\frac{2\dot{\varphi}}{n}\left(
\frac{(n+1)}{1-|D\varphi|^{2}}\sigma^{kl}  -
\iota^{kl}\right)\varphi_{l}\varphi_{k11} =
-\frac{2\dot{\varphi}}{n}\left( \frac{(n+1)}{1-|D\varphi|^{2}
}\sigma^{kl}- \iota^{kl} \right)\varphi_{l}\left( \iota_{11k} -
\varphi_{k} + \sigma_{1k}\varphi_{1} + 2\varphi_{1}\varphi_{1k}
\right)
\end{equation*}
in view of \eqref{f3}, which implies
\begin{equation*}
\begin{aligned}
&\quad -\frac{2\dot{\varphi}}{n}\left(
\frac{(n+1)}{1-|D\varphi|^{2}}\sigma^{kl} - \iota^{kl} \right)
\varphi_{l}\varphi_{k11} - Q^{k}\iota_{11,k}\\
&= -\frac{2\dot{\varphi}}{n}\left(
\frac{(n+1)}{1-|D\varphi|^{2}}\sigma^{kl} - \iota^{kl}
\right)\varphi_{l} \left( -\varphi_{k} + \sigma_{1k}\varphi_{1} +
2\varphi_{1}\varphi_{1k} \right).
\end{aligned}
\end{equation*}
Therefore,
\begin{equation*}
\begin{aligned}
\mathcal{L}\iota_{11} &=
\frac{(\dot{\varphi}_{1})^{2}}{\dot{\varphi}} -
\frac{\dot{\varphi}}{n}\iota^{kl}_{,1}\iota_{kl,1} -
\frac{2}{n}\dot{\varphi}\tr\iota^{kl}\varphi_{11} -
\frac{2}{n}\dot{\varphi}\left( \tr\iota^{kl} - n -
\iota^{kl}\varphi_{k}\varphi_{l} \right)\\
&\quad
-\frac{4(n+1)\dot{\varphi}}{n}\frac{1}{(1-|D\varphi|^{2})^{2}}(\sigma^{kl}\varphi_{k}\varphi_{l1})^{2}
-\frac{2(n+1)\dot{\varphi}}{n}\frac{1}{1-|D\varphi|^{2}}\sigma^{kl}\varphi_{k1}\varphi_{l1}\\
&\quad -\frac{2\dot{\varphi}}{n}\frac{(n+1)}{1-|D\varphi|^{2}}\left(
(\varphi_{1})^{2} - |D\varphi|^{2} \right)
-\frac{2}{n}\dot{\varphi}\iota^{kl}\left(
-\sigma_{kl}(\varphi_{1})^{2} + \varphi_{k}\varphi_{l} \right),
\end{aligned}
\end{equation*}
which implies
\begin{equation*}\label{f9}
\begin{aligned}
\mathcal{L}\iota_{11}&=\frac{(\dot{\varphi}_{1})^{2}}{\dot{\varphi}} - \frac{1}{n}\dot{\varphi}\iota^{kl}_{,1}\iota_{kl,1} - \frac{4(n+1)\dot{\varphi}}{n}\frac{1}{(1-|D\varphi|^{2})^{2}}(\sigma^{kl}\varphi_{k}\varphi_{l1})^{2}\\
&\quad
-\frac{2(n+1)\dot{\varphi}}{n}\frac{1}{1-|D\varphi|^{2}}\sigma^{kl}\varphi_{k1}
\varphi_{l1} - \frac{2\dot{\varphi}}{n}\frac{(n+1)}{1-|D\varphi|^{2}}\left( (\varphi_{1})^{2} - |D\varphi|^{2} \right)\\
&\quad - \frac{2}{n}\dot{\varphi}(\iota_{11}\tr\iota^{kl} - n).
\end{aligned}
\end{equation*}

Finally, we prove the third equality. Differentiating the function
$\tau^{l}\varphi_{l}$ twice with $x\in M^{n}$, we have
$$(\tau^{l}\varphi_{l})_{i} = \tau^{l}_{,i}\varphi_{l} + \tau^{l}\varphi_{li},$$
and
$$(\tau^{l}\varphi_{l})_{ij} = \tau^{l}_{,ij}\varphi_{l} + \tau^{l}_{,i}\varphi_{lj}
+ \tau^{l}_{,j}\varphi_{li} + \tau^{l}\varphi_{lij}.$$
Differentiating $\tau^{l}\varphi_{l}$ w.r.t. $t$ yields
\begin{equation*}
\begin{aligned}
(\tau^{l}\varphi_{l})_{t} &= \varphi_{lt}\tau^{l}\\
&=
\left(Q^{ij}\varphi_{ijl} + Q^{k}\varphi_{kl}\right)\tau^{l}\\
&= \left( Q^{ij}\varphi_{lij} + Q^{ij}\varphi_{l}\sigma_{ij} -
Q^{ij}\varphi_{j}\sigma_{il} + Q^{k}\varphi_{kl} \right)\tau^{l}.
\end{aligned}
\end{equation*}
Therefore,
\begin{equation*}
\begin{aligned}
\mathcal{L}(\varphi_{l}\tau^{l}) &=
(\varphi_{l}\tau^{l})_{t} - Q^{ij}(\varphi_{l}\tau^{l})_{ij} - Q^{k}(\varphi_{l}\tau^{l})_{k}\\
&= Q^{ij}\left( \sigma_{ij}\varphi_{l}\tau^{l} -
\sigma_{il}\varphi_{j}\tau^{l} - 2\varphi_{li}\tau^{l}_{,j} -
\varphi_{l}\tau^{l}_{,ij} \right) - Q^{k}\varphi_{l}\tau^{l}_{,k},
\end{aligned}
\end{equation*}
which implies (\ref{evolu 3}) by inserting the expressions of
$Q^{ij}$ and $Q^{k}$ directly.
\end{proof}

\footnote{~This assumption on $\mu$ is just to conveniently judge
the sign of some terms below, which does not conflict with the
property required in Theorem \ref{main1.1}. Besides, without loss of
generality, one can also require that $\mu$ is
future-directed.}Assume further that $\mu$ is a smooth extension of
the future-directed spacelike unit normal to $\partial M^{n}$ that
vanishes outside a tubular neighborhood of $\partial M^{n}$. We
define for $(x, \xi_{1}, \xi_{2},
t)\in\overline{M^{n}}\times\mathbb{R}^{n}\times\mathbb{R}^{n}\times[0,T]$,
$$\eta(x, \xi_{1}, \xi_{2}, t)=\mu^{i}_{,j}\varphi_{i}\left(\langle\xi_{1}, \mu\rangle \zeta^{j}_{2} + \langle\xi_{2}, \mu\rangle \zeta^{j}_{1} \right),$$
where
$$\zeta_{i}=\xi_{i} - \langle\xi_{i}, \mu\rangle \mu$$
indicates the tangential component of the vector $\xi_{i}$, with
$i=1,2$, and where $\langle\cdot, \cdot\rangle$ is the inner product
induced by $\sigma$. Moreover, let $\eta_{ij}(x,t):
\overline{M^{n}}\times[0,T]\to\mathbb{R}^{n}$, with $1\leq i, j\leq
n$, represent the component functions
$$\eta_{ij}(x,t) = \mu^{q}_{,p}\varphi_{q}\left[\sigma_{ki}\mu^{k}(\delta^{p}_{j} - \sigma_{lj}\mu^{l}\mu^{p})
+ \sigma_{kj}\mu^{k}(\delta^{p}_{i} -
\sigma_{li}\mu^{l}\mu^{p})\right]$$ of the symmetric $(0,2)$-tensor
field $\eta$.

\begin{remark}
\rm{$\eta(x, \xi_{1}, \xi_{2}, t)$ is not an important part in the
following interior estimate, but will paly a great role in the
non-normal boundary estimate below.}
\end{remark}

We define a function as follows
$$S(x,\xi, t)=\log\left(\frac{[\iota_{ij}(x,t) + \eta_{ij}(x,t)]\xi^{i}\xi^{j}}{\sigma_{ij}\xi^{i}\xi^{j}}
+ C\right) + \frac{1}{2}\lambda|D\varphi|^{2}$$ for $(x, \xi,
t)\in\overline{M^{n}}\times\mathscr{H}^{n}(1)\times[0,T]$, where $C$
and $\lambda$ are constants which will be chosen later.

\begin{proposition}\label{C2 interior estimate}
Let $\varphi$ be a solution of the flow \eqref{Evo-1}, and assume
that $S$ attains its maximum in
$M^{n}\times\mathscr{H}^{n}(1)\times[0,T]$ for some fixed $T<T^{*}$.
Then, there exists $C=C(n, M^{n}_{0})$ such that
$$\varphi_{ij}\xi^{i}\xi^{j} \leq C(n, M^{n}_{0}),
~~\qquad \forall (x, \xi,
t)\in\overline{M^{n}}\times\mathscr{H}^{n}(1)\times[0,T].$$
\end{proposition}

\begin{proof}
Assume that $S(x, \xi, t)$ achieves its maximum at $(x_{0}, \xi_{0},
t)\in M^{n}\times\mathscr{H}^{n}(1)\times[0,T]$. Choose Riemannian
normal coordinates at $x_{0}$ such that at this point we have
$$\sigma_{ij}(x_{0}) = \delta_{ij}, ~~\qquad \partial_{k}\sigma_{ij}(x_{0}) = 0.$$
And we further rotate the coordinate system at $(x_{0}, t_{0})$ such
that the matrix $\iota_{ij} + \eta_{ij}$ is diagonal, i.e.,
$$\iota_{ij} + \eta_{ij} = (\iota_{ii} + \eta_{ii})\delta_{ij},$$
with
$$\iota_{nn} + \eta_{nn} \leq \cdots \leq \iota_{22} + \eta_{22} \leq \iota_{11} + \eta_{11}.$$
Thus, since the matrix $\iota_{ij}$ is positive define, we have at
$(x_{0}, t_{0})$,
\begin{equation}\label{f10}
|\iota_{ii}|\leq \iota_{11} + C(M_{0}^{n}) \quad \mathrm{and} \quad
|\iota_{ij}| \leq C(M_{0}^{n}) ~~\quad \mathrm{for}~~i\neq j
\end{equation}
in view of the $C^{1}$-estimate \eqref{Gra-est}. Set $\xi_{1}(x)$
around a neighbor of $x_{0}$ and $\xi_{1}(x_{0}) = \xi_{0}$. At
$(x_{0}, t_{0})$, there holds
\begin{equation*}
\iota_{11} + \eta_{11} = \sup\limits_{\xi\in\mathscr{H}^{n}(1)}
\frac{\left[ \iota_{ij}(x, t) + \eta_{ij}(x, t)
\right]\xi^{i}\xi^{j}} {\sigma_{ij}\xi^{i}\xi^{j}},
\end{equation*}
and in a neighborhood of $(x_{0}, t_{0})$
\begin{equation*}
\iota_{11} + \eta_{11} \leq \sup\limits_{\xi\in\mathscr{H}^{n}(1)}
\frac{\left[ \iota_{ij}(x, t) + \eta_{ij}(x, t)
\right]\xi^{i}\xi^{j}} {\sigma_{ij}\xi^{i}\xi^{j}}.
\end{equation*}
Furthermore, it's easy to check that the covariant (at least up to
the second order) and the first time derivatives of
\begin{equation*}
\frac{\left[ \iota_{ij}(x, t) + \eta_{ij}(x, t)
\right]\xi_{1}^{i}\xi_{1}^{j}} {\sigma_{ij}\xi_{1}^{i}\xi_{1}^{j}}
\end{equation*}
and
$$\iota_{11} + \eta_{11}$$
do coincide at $(x_{0}, t_{0})$ (in normal coordinate). Without loss
of generality, we treat $\iota_{11} + \eta_{11}$ as a scalar and
pretend that $W$ is defined by
\begin{equation*}
W(x,t) = \log (\iota_{11} + \eta_{11} + C) +
\frac{1}{2}\lambda|D\varphi|^{2},
\end{equation*}
which achieves its maximum at $(x_{0}, t_{0})\in M^{n}\times[0, T]$.
Here, noticing that we can choose $C(M_{0}^{n})$ large enough
satisfying
\begin{equation}\label{f11}
0\leq \eta_{11} + C(M_{0}^{n}),
\end{equation}
since $\eta_{11}$ is bounded by the $C^{1}$-estimate
\eqref{Gra-est}.

In the following, we want to compute
\begin{equation*}
\begin{aligned}
\mathcal{L}W &=
\dot{W} - Q^{ij}W_{ij} - Q^{k}W_{k}\\
&= \mathcal{L}\left( \log(\iota_{11} + \eta_{11} + C) \right) +
\frac{1}{2}\lambda\mathcal{L}(|D\varphi|^{2}).
\end{aligned}
\end{equation*}
First, after a simple calculation, we can rewrite the first term of
the RHS of the above equality as follows
\begin{equation*}
\begin{aligned}
&\quad
\mathcal{L}\left( \log(\iota_{11} + \eta_{11} + C) \right)\\
&= \frac{\mathcal{L}\iota_{11}}{\iota_{11} + \eta_{11} + C} +
\frac{\mathcal{L}\eta_{11}}{\iota_{11} + \eta_{11} + C} -
\frac{1}{n}\dot{\varphi}\iota^{ij} \frac{(\iota_{11,i} +
\eta_{11,i})(\iota_{11,j} + \eta_{11,j})}{(\iota_{11} + \eta_{11} +
C)^{2}}.
\end{aligned}
\end{equation*}
Now, we begin to estimate $\mathcal{L}\iota_{11}$ through the
evolution equation \eqref{evolu 2}. Using the Cauchy-Schwarz
inequality, one has
\begin{equation}\label{f12}
\begin{aligned}
&\quad -\frac{2(n+1)\dot{\varphi}}{n}\frac{1}{1-|D\varphi|^{2}}
\left(
\frac{2}{1-|D\varphi|^{2}}(\sigma^{kl}\varphi_{k}\varphi_{l1})^{2}
+ \sigma^{kl}\varphi_{k1}\varphi_{l1} \right)\\
&\leq -\frac{2(n+1)\dot{\varphi}}{n}\frac{1}{(1-|D\varphi|^{2})^{2}}
\sigma^{kl}\varphi_{k1}\varphi_{l1}.
\end{aligned}
\end{equation}
On the other hand,
\begin{equation*}
\sigma^{kl}\varphi_{k1}\varphi_{l1} =
\sigma^{kl}\iota_{k1}\iota_{l1} + \sigma_{11} - 2\iota_{11} -
2(\varphi_{1})^{2}
 + 2\varphi_{1}\sigma^{kl}\iota_{k1}\varphi_{l}
 + (\varphi_{1})^{2}|D\varphi|^{2}.
\end{equation*}
Using \eqref{f10}, together with the $C^{1}$-estimate
\eqref{Gra-est}, the inequality \eqref{f12} becomes
\begin{equation}\label{f13}
\begin{aligned}
&\quad -\frac{2(n+1)\dot{\varphi}}{n}\frac{1}{1-|D\varphi|^{2}}
\left(
\frac{2}{1-|D\varphi|^{2}}(\sigma^{kl}\varphi_{k}\varphi_{l1})^{2}
+ \sigma^{kl}\varphi_{k1}\varphi_{l1} \right)\\
&\leq -\frac{2(n+1)\dot{\varphi}}{n}\frac{1}{(1-|D\varphi|^{2})^{2}}
\left(\sigma^{kl}\iota_{k1}\iota_{l1} - C(M_{0}^{n})\iota_{11} +
C(M_{0}^{n}, \rho)\right).
\end{aligned}
\end{equation}
Inserting \eqref{f13} into $\mathcal{L}\iota_{11}$, abandoning the
non-positive terms and using the $C^{1}$-estimate \eqref{Gra-est}
again, we can obtain
\begin{equation}\label{f14}
\mathcal{L}\iota_{11}\leq
-\frac{2(n+1)\dot{\varphi}}{nC_{\rho}^{2}}\sigma^{kl}\iota_{k1}\iota_{l1
}- C(n, M_{0}^{n}, \rho)\dot{\varphi}(\iota_{11}\tr\iota^{kl} -
\iota_{11} + 1) -
\frac{1}{n}\dot{\varphi}\iota^{kl}_{,1}\iota_{kl,1},
\end{equation}
where $C_{\rho}:= 1-\rho^{2}$.

Next, recalling the equation \eqref{evolu 1} and using the
$C^{1}$-estimate \eqref{Gra-est},
\begin{equation}\label{f15}
\begin{aligned}
&\quad
\lambda\mathcal{L}(\frac{1}{2}|D\varphi|^{2})\\
&=
\frac{\lambda}{n}\dot{\varphi}\left((1-|D\varphi|^{2})\iota^{ij}\sigma_{ij}
- (1-|D\varphi|^{2})\iota^{ij}\varphi_{i}\varphi_{j} +
\iota_{ij}\sigma^{ij}
+ 2|D\varphi|^{2} - 2n \right)\\
&\leq \frac{\lambda}{n}\dot{\varphi} \left( \iota_{ij}\sigma^{ij} -
(1-|D\varphi|^{2})\iota^{ij}\varphi_{i}\varphi_{j} + 2|D\varphi|^{2}
\right) + \frac{C_{\rho}\lambda}{n}\dot{\varphi}\tr\iota^{ij} -
2\lambda\dot{\varphi}.
\end{aligned}
\end{equation}
Then, it follows from \eqref{f14} and \eqref{f15} that
\begin{equation*}
\begin{aligned}
\mathcal{L}W &\leq \frac{1}{\iota_{11} + \eta_{11} + C} \left[
-\frac{2(n+1)\dot{\varphi}}{C_{\rho}^{2}n}\sigma^{kl}\iota_{k1}\iota_{l1}
- C(n, M_{0}^{n}, \rho)\dot{\varphi}(\iota_{11}\tr\iota^{kl} -
\iota_{11} + 1)
 - \frac{1}{n}\dot{\varphi}\iota^{kl}_{,1}\iota_{kl,1} \right]\\
 &\quad
  + \frac{\mathcal{L}\eta_{11}}{\iota_{11} + \eta_{11} + C} - \frac{1}{n}\dot{\varphi}\iota^{ij}\frac{(\iota_{11,i} + \eta_{11,i})(\iota_{11,j} + \eta_{11,j})}{(\iota_{11} + \eta_{11} + C)^{2}}\\
  &\quad
  +\frac{\lambda}{n}\dot{\varphi}\left( \iota_{ij}\sigma^{ij} - (1-|D\varphi|^{2})\iota^{ij}\varphi_{i}\varphi_{j} + 2|D\varphi|^{2} \right)
  + \frac{C_{\rho}\lambda}{n}\dot{\varphi}\tr\iota^{ij} - 2\lambda\dot{\varphi}.
\end{aligned}
\end{equation*}
To make progress, we need to estimate
\begin{equation*}
\begin{aligned}
&\quad -\frac{1}{\iota_{11} + \eta_{11} +
C}\frac{2(n+1)\dot{\varphi}}{C_{\rho}^{2}n}
\sigma^{kl}\iota_{k1}\iota_{l1} + \frac{\lambda\dot{\varphi}}{n}\iota_{ij}\sigma^{ij}\\
&\leq -C(M_{0}^{n})\dot{\varphi}\left( \frac{(\iota_{11} +
C)^{2}}{\iota_{11} + \eta_{11} + C}
 - \lambda\iota_{11} \right)\\
 &\leq
-C(M_{0}^{n})(1-\lambda)\dot{\varphi}\iota_{11}
\end{aligned}
\end{equation*}
in view of \eqref{f10} and \eqref{f11}, where we assume that
$\iota_{11}\geq 1$. Otherwise, $\iota_{11}$ is bounded from above
and our theorem holds true.

Now, we only leave the term $\mathcal{L}\eta_{11}$ to be estimated.
Clearly, $\eta_{11}$ can be written as
$$\eta_{11} = \tau^{l}\varphi_{l} + C,$$
where $\tau^{l}: \overline{M^{n}} \to \mathbb{R}$ does not depend on
$\varphi$. Recalling the equation \eqref{evolu 3}, we have
\begin{equation*}
\begin{aligned}
\mathcal{L}(\tau^{l}\varphi_{l}) &=
 -\frac{1}{n}\dot{\varphi}\iota^{ij}\left( \sigma_{ij}\varphi_{l}\tau^{l} - \sigma_{il}\varphi_{j}\tau^{l} - 2\varphi_{li}\tau^{l}_{,j} - \varphi_{l}\tau^{l}_{,ij} + 2\varphi_{j}\varphi_{l}\tau^{l}_{,i} \right)\\
&\quad
 + \frac{2\dot{\varphi}}{n}\frac{(n+1)}{1-|D\varphi|^{2}}
 \varphi^{k}\varphi_{l}\tau^{l}_{,k},
\end{aligned}
\end{equation*}
by applying
$$\iota^{ij}\varphi_{li} = \delta^{j}_{l} - \iota^{ij}\sigma_{li} + \iota^{ij}\varphi_{l}\varphi_{i}.$$
We can obtain by the $C^{1}$-estimate \eqref{Gra-est} that
$$\mathcal{L}\eta_{11} \leq -C(n, M_{0}^{n}, \rho)\dot{\varphi}(\tr\iota^{ij} + 1).$$
Therefore,
\begin{equation*}
\begin{aligned}
\mathcal{L}W &\leq -\frac{\dot{\varphi}}{\iota_{11} + \eta_{11} +C}
\left( C(n, M_{0}^{n}, \rho)\iota_{11}\tr\iota^{kl} - C(n, M_{0}^{n}, \rho)\iota_{11} +C(n, M_{0}^{n}, \rho) + \frac{1}{n}\iota^{kl}_{,1}\iota_{kl,1} \right) \\
&\quad - \frac{\dot{\varphi}}{\iota_{11} + \eta_{11} + C} \left(
C(n, M_{0}^{n}, \rho)\tr\iota^{ij} + C(n, M_{0}^{n}, \rho) \right) -
\frac{1}{n}\dot{\varphi}\iota^{ij}
\frac{(\iota_{11,i} + \eta_{11,i})(\iota_{11,j} + \eta_{11,j})}{(\iota_{11} + \eta_{11} + C)^{2}} \\
&\quad - \frac{\lambda\dot{\varphi}}{n} \left[
(1-|D\varphi|^{2})\iota^{ij}\varphi_{i}\varphi_{j} - 2|D\varphi|^{2}
\right] + \frac{C_{\rho}\lambda}{n}\dot{\varphi}\tr\iota^{ij} -
2\lambda\dot{\varphi} -
C(M_{0}^{n})(1-\lambda)\dot{\varphi}\iota_{11}.
\end{aligned}
\end{equation*}
The last term left which we have to estimate is
$$-\frac{1}{n}\dot{\varphi}\left( \frac{1}{\iota_{11} + \eta_{11} + C}\iota^{kl}_{,1}\iota_{kl,1} + \iota^{ij}\frac{(\iota_{11,i} + \eta_{11,i})(\iota_{11,j}
+ \eta_{11,j})}{(\iota_{11} + \eta_{11} + C)^{2}} \right).$$
 For
convenience later, set $V=\iota_{11} + \eta_{11} + C$. Then
\begin{equation*}
\begin{aligned}
&\quad
\frac{1}{V}\iota^{kl}_{,1}\iota_{kl,1} + \iota^{kl}\frac{V_{k}V_{l}}{V^{2}}\\
&=
-\frac{1}{V}\iota^{pk}\iota^{ql}\iota_{pq,1}\iota_{kl,1} + \iota^{kl}\frac{V_{k}V_{l}}{V^{2}}\\
&\leq
-\frac{1}{V}\frac{1}{\iota_{11}}\iota^{kl}\iota_{1k,1}\iota_{1l,1} + \iota^{kl}\frac{V_{k}V_{l}}{V^{2}}\\
&= \iota^{kl}\frac{V_{k}V_{l}}{V\iota_{11}} -
\frac{1}{V}\frac{1}{\iota_{11}}\iota^{kl}\iota_{1k,1}\iota_{1l,1} -
\frac{\eta_{11} + C}{V^{2}\iota_{11}}\iota^{kl}V_{k}V_{l}.
\end{aligned}
\end{equation*}
In view of \eqref{f11}, together with the fact that the matrix
$(\iota^{kl})$ is positive definite, we have
$$-\frac{\eta_{11} + C}{V^{2}\iota_{11}}\iota^{kl}V_{k}V_{l}\leq 0.$$
Thus,
$$\frac{1}{V}\iota^{kl}_{,1}\iota_{kl,1} + \iota^{kl}\frac{V_{k}V_{l}}{V^{2}}
\leq \frac{1}{V\iota_{11}}(\iota^{kl}V_{k}V_{l} -
\iota^{kl}\iota_{1k,1}\iota_{1l,1}).$$
  Recalling that
$$\iota^{kl}V_{k}V_{l} = \iota^{kl}(\iota_{11,k}\iota_{11,l} + 2\iota_{11,k}\eta_{11,l} + \eta_{11,k}\eta_{11,l}),$$
 and then
it follows from the equality \eqref{f2}
 that
\begin{equation}\label{f16}
\begin{aligned}
&\quad
\frac{1}{V}\iota^{kl}_{,1}\iota_{kl,1} + \iota^{kl}\frac{V_{k}V_{l}}{V^{2}}\\
&\leq \frac{1}{V\iota_{11}}\left(
2\iota^{kl}\iota_{11,k}\varphi_{l}\iota_{11} -
2\iota_{11,1}\varphi_{1} -
(\iota_{11})^{2}\iota^{kl}\varphi_{k}\varphi_{l} +
\iota_{11}(\varphi_{1})^{2}
\right.\\
 &\quad
 \left.
 + 2\iota^{kl}\iota_{11,k}\eta_{11,l} + \iota^{kl}\eta_{11,k}\eta_{11,l} \right).
\end{aligned}
\end{equation}
Since $W(x,t)$ achieves its maximum at $(x_{0}, t_{0})\in
M^{n}\times[0,T]$, so $W_{i} = 0$ implies
$$W_{i} = \frac{V_{i}}{V} + \lambda\sigma^{kl}\varphi_{k}\varphi_{li} = 0.$$
Therefore,
\begin{equation}\label{f17}
\iota_{11,1} = -\lambda V\sigma^{kl}\varphi_{k}\varphi_{1l} -
\eta_{11,1}
\end{equation}
and
\begin{equation}\label{f18}
\begin{aligned}
\iota^{kl}\iota_{11,k} &= \iota^{kl}(-\lambda V\sigma^{pq}\varphi_{pk}\varphi_{q} - \eta_{11,k})\\
&=
-\lambda V\iota^{kl}\varphi_{pk}\sigma^{pq}\varphi_{q} - \iota^{kl}\eta_{11,k}\\
&=
-\lambda V(\delta^{l}_{p} - \iota^{kl}\sigma_{pk} + \iota^{kl}\varphi_{p}\varphi_{k})\sigma^{pq}\varphi_{q} - \iota^{kl}\eta_{11,k}\\
&= \lambda V\iota^{kl}\varphi_{k}(1-|D\varphi|^{2}) - \lambda
V\sigma^{lp}\varphi_{p} - \iota^{kl}\eta_{11,k}.
\end{aligned}
\end{equation}
Then,  by the $C^{1}$-estimate \eqref{Gra-est}, \eqref{f10},
\eqref{f17} and \eqref{f18}, we have
\begin{equation*}
\begin{aligned}
-2\iota_{11,1}\varphi_{1} &= 2\lambda V\sigma^{kl}\varphi_{k}\varphi_{1}\varphi_{1l} + 2\eta_{11,1}\varphi_{1}\\
&=
2\lambda V\sigma^{kl}\iota_{1l}\varphi_{1}\varphi_{k} - 2\lambda V(\varphi_{1})^{2} + 2\lambda V|D\varphi|^{2}(\varphi_{1})^{2} + 2(\tau^{l}\varphi_{l})_{1}\varphi_{1}\\
&=
2\lambda V\sigma^{kl}\iota_{1l}\varphi_{1}\varphi_{k} - 2\lambda V(\varphi_{1})^{2} + 2\lambda V|D\varphi|^{2}(\varphi_{1})^{2} + 2\tau^{l}_{,1}\varphi_{l}\varphi_{1}\\
&\quad
+ 2\tau^{l}\iota_{l1}\varphi_{1} - 2\tau^{1}\varphi_{1} + 2\tau^{l}\varphi_{l}(\varphi_{1})^{2}\\
&\leq C(M_{0}^{n}, \rho)\lambda V(\iota_{11} + 1)+ C(M_{0}^{n},
\rho)(\iota_{11} + 1)
\end{aligned}
\end{equation*}
and
\begin{equation*}
\begin{aligned}
&\quad \frac{1}{V\iota_{11}}\left(
2\iota^{kl}\iota_{11,k}\varphi_{l}\iota_{11} -
(\iota_{11})^{2}\iota^{kl}\varphi_{k}\varphi_{l} + \iota_{11}(\varphi_{1}^{2}) \right)\\
&\leq 2\lambda(1-|D\varphi|^{2})\iota^{kl}\varphi_{k}\varphi_{l} -
\frac{2}{V}\iota^{kl}\varphi_{l}\eta_{11,k} + \frac{C(\rho)}{V}.
\end{aligned}
\end{equation*}
Thus, combining the above inequalities and the assumption
$\iota_{11}\geq1$, we have
\begin{equation*}
\begin{aligned}
&\quad \frac{1}{V\iota_{11}}\left(
2\iota^{kl}\iota_{11,k}\varphi_{l}\iota_{11}
 - 2\iota_{11,1}\varphi_{1} - (\iota_{11})^{2}\iota^{kl}\varphi_{k}\varphi_{l}
 + \iota_{11}(\varphi_{1})^{2} \right)\\
 &\leq
 2\lambda(1-|D\varphi|^{2})\iota^{kl}\varphi_{k}\varphi_{l} - \frac{2}{V}\iota^{kl}\varphi_{l}\eta_{11,k} + \frac{C(\rho)}{V} + C(M_{0}^{n}, \rho)\lambda\frac{\iota_{11} + 1}{\iota_{11}} + C(M_{0}^{n}, \rho)\frac{\iota_{11} + 1}{V\iota_{11}}\\
 &\leq
 2\lambda(1-|D\varphi|^{2})\iota^{kl}\varphi_{k}\varphi_{l} + \frac{C(M_{0}^{n}, \rho)}{V}(\tr\iota^{kl} + 1) + \frac{C(\rho)}{V} + C(M_{0}^{n}, \rho)\lambda,
\end{aligned}
\end{equation*}
where we have used the following inequality
\begin{equation}\label{f19}
\begin{aligned}
\iota^{kl}\varphi_{l}\eta_{11,k} &= \iota^{kl}\varphi_{l}
(\tau^{i}\varphi_{ik} + \tau^{i}_{,k}\varphi_{i})\\
&= (\delta^{l}_{i} - \iota^{kl}\sigma_{ik} +
\iota^{kl}\varphi_{i}\varphi_{k})
\varphi_{l}\tau^{i} + \iota^{kl}\varphi_{l}\varphi_{i}\tau^{i}_{,k}\\
&\leq C(M_{0}^{n}, \rho)(\tr\iota^{kl} + 1)
\end{aligned}
\end{equation}
to get the second inequality. Now, we estimate the last two terms in
the bracket of RHS of \eqref{f16}. Using \eqref{f18} and the
assumption $\lambda\geq 1$ and $\iota_{11}\geq 1$, we have
\begin{equation*}
\begin{aligned}
&\quad
\frac{1}{V\iota_{11}}\left( 2\iota^{kl}\iota_{11,k}\eta_{11,l} + \iota^{kl}\eta_{11,k}\eta_{11,l} \right)\\
&=
\frac{1}{V\iota_{11}}\left( 2\lambda V\iota^{kl}\varphi_{k}(1-|D\varphi|^{2})\eta_{11,l} - 2\lambda V\sigma^{lp}\varphi_{p}\eta_{11,l} - \iota^{kl}\eta_{11,k}\eta_{11,l} \right)\\
&\leq
\frac{1}{V\iota_{11}}\left( \lambda V C(M_{0}^{n}, \rho)(\tr\iota^{kl} + 1) + C(M_{0}^{n}, \rho)\lambda V(\iota_{11} + 1) \right)\\
&\leq C(M_{0}^{n}, \rho)\lambda\left(
\frac{\tr\iota^{kl}}{\iota_{11}} + 1\right)
\end{aligned}
\end{equation*}
in view of \eqref{f19}, and
\begin{equation*}
\begin{aligned}
\sigma^{lp}\varphi_{p}\eta_{11,l} &=
\sigma^{lp}\varphi_{p}(\tau^{i}\varphi_{il} + \tau^{i}_{,l}\varphi_{i})\\
&=
\sigma^{lp}\iota_{il}\varphi_{p}\tau^{i} - (1-|D\varphi|^{2})\varphi_{i}\tau^{i} + \sigma^{lp}\varphi_{p}\varphi_{i}\tau^{i}_{,l}\\
&\leq C(M_{0}^{n}, \rho)(\iota_{11} + 1).
\end{aligned}
\end{equation*}
Inserting the above equality and \eqref{f17} into \eqref{f16}, we
get at $(x_{0}, t_{0})$,
\begin{equation*}\label{f20}
\begin{aligned}
&\quad
\frac{1}{V}\iota^{kl}_{,1}\iota_{kl,1} + \iota^{kl}\frac{V_{k}V_{l}}{V^{2}}\\
&\leq 2\lambda(1-|D\varphi|^{2})\iota^{kl}\varphi_{k}\varphi_{l} +
\frac{C(M_{0}^{n}, \rho)}{V}(\tr\iota^{kl} + 1) + C(M_{0}^{n},
\rho)\lambda\left( \frac{\tr\iota^{kl}}{\iota_{11}} + 1 \right).
\end{aligned}
\end{equation*}
Thus,
\begin{equation*}
\begin{aligned}
\mathcal{L}W &\leq -\frac{\dot{\varphi}}{V}\left( C(n, M_{0}^{n},
\rho)\iota_{11}\tr\iota^{kl} - C(n, M_{0}^{n}, \rho)\iota_{11}
+ C(n, M_{0}^{n}, \rho) \right) \\
&\quad - \frac{\dot{\varphi}}{V}\left(C(n, M_{0}^{n},
\rho)\tr\iota^{ij} + C(n, M_{0}^{n}, \rho)\right)
-\frac{\lambda\dot{\varphi}}{n}\left[ (1-|D\varphi|^{2})\iota^{ij}\varphi_{i}\varphi_{j} - 2|D\varphi|^{2} \right]\\
&\quad
+ \frac{C_{\rho}\lambda}{n}\dot{\varphi}\tr\iota^{ij} - 2\lambda\dot{\varphi} - C(M_{0}^{n})(1-\lambda)\dot{\varphi}\iota_{11}\\
&\quad
- \frac{\dot{\varphi}}{n}\left[ 2\lambda(1-|D\varphi|^{2})\iota^{kl}\varphi_{k}\varphi_{l} + \frac{C(M_{0}^{n}, \rho)}{V}(\tr\iota^{kl} + 1) + C(M_{0}^{n}, \rho)\lambda(\frac{\tr\iota^{kl}}{\iota_{11}} + 1) \right]\\
&\leq
-\frac{\dot{\varphi}}{V}C(n, M_{0}^{n}, \rho)\left( \iota_{11}(\tr\iota^{kl} - 1) + 1 \right) - \frac{\lambda\dot{\varphi}}{n}\left[ 3(1-|D\varphi|^{2})\iota^{ij}\varphi_{i}\varphi_{j} - 2|D\varphi|^{2} \right]\\
&\quad +\frac{C_{\rho}\lambda}{n}\dot{\varphi}\tr\iota^{ij} -
2\lambda\dot{\varphi} -
C(M_{0}^{n})(1-\lambda)\dot{\varphi}\iota_{11}
-\frac{\dot{\varphi}}{n}C(M_{0}^{n}, \rho)\lambda\left( \frac{\tr\iota^{kl}}{\iota_{11}} + 1\right)\\
&\leq -\frac{\dot{\varphi}}{V}C(n, M_{0}^{n}, \rho)\left(
\iota_{11}(\tr\iota^{kl} - 1) + 1\right) -
\frac{C(M_{0}^{n})\lambda\dot{\varphi}}{n}(\iota_{11} + 1) + \frac{C_{\rho}\lambda}{n}\dot{\varphi}\tr\iota^{ij}\\
&\quad
-2\lambda\dot{\varphi} - C(M_{0}^{n})\dot{\varphi}(1-\lambda)\iota_{11} - \frac{\dot{\varphi}}{n}C(M_{0}^{n}, \rho)\lambda\left( \frac{\tr\iota^{kl}}{\iota_{11}} + 1\right)\\
&\leq -\dot{\varphi}\tr\iota^{kl}\left( \frac{C(M_{0}^{n},
\rho)\lambda}{\iota_{11}} + C(n, M_{0}^{n}, \rho)\iota_{11} -
\frac{C_{\rho}\lambda}{n} \right) - C(n, M_{0}^{n},
\rho)\dot{\varphi}\left(1 + \lambda + (1 -
\frac{1}{2}\lambda)\iota_{11}\right).
\end{aligned}
\end{equation*}
Since $\dot{\varphi}<0$, we take $\lambda$ and $\iota_{11}$ large
enough such that $\frac{C(M_{0}^{n}, \rho)\lambda}{\iota_{11}} +
C(n, M_{0}^{n}, \rho)\iota_{11} - \frac{C_{\rho}\lambda}{n} \leq 0$
(otherwise, $\iota_{11}$ is bounded from above and our theorem holds
true). Then, in view of $\mathcal{L}W \geq 0$, we can obtain
$$\iota_{11} \leq C(n, M_{0}^{n}).$$
Therefore, we conclude that $\iota_{11}$ has an upper bound, which
implies that the second covariant derivatives of $\varphi$ is
bounded from above. This completes the proof.
\end{proof}

\subsection{Double normal $C^{2}$ boundary estimates.}
\ Let

\begin{equation*}
\begin{aligned}
\widetilde{\mathcal{L}}U &= \dot{U} - Q^{ij}U_{ij} +
\frac{2(n+1)}{n}\frac{\dot{\varphi}}{1-|D\varphi|^{2}}\varphi^{k}U_{k}\\
&= \dot{U} + \frac{1}{n}\dot{\varphi}\iota^{ij}U_{ij} +
\frac{2(n+1)}{n}\frac{\dot{\varphi}}{1-|D\varphi|^{2}}\varphi^{k}U_{k}
\end{aligned}
\end{equation*}
and
\begin{equation*}
q(x) = d(x) - \vartheta d^{2}(x).
\end{equation*}
Here, $d$ denotes the distance to $\partial M^{n}$, which is smooth
function in $M_{n}^{\delta} = \{ x\in M^{n}:\mathrm{dist}(x,
\partial M^{n}) < \delta \}$ for $\delta$ small enough, and
$\vartheta$ denotes a constant to be chosen sufficiently large.
Thus, $q: M_{n}^{\delta} \to \mathbb{R}$ is a smooth function.

To derive double normal $C^{2}$ boundary estimates, we need the
following lemma.

\begin{lemma}\label{lemma 4.5}
For any solution $\varphi$ of the flow \eqref{Evo-1}, we can choosen
$\vartheta$ so large and $\delta$ so small such that

\begin{equation*}
\widetilde{\mathcal{L}}q(x) \geq
-\frac{1}{4n}\kappa_{0}\dot{\varphi}\tr(\iota^{ij})   ~\qquad in
~\quad M_{n}^{\delta},
\end{equation*}
where $\kappa_{0}$ is a positive constant depending on $\partial
M^{n}$.
\end{lemma}

\begin{proof}
Differentiating the function $q$ twice w.r.t. $x$ yields
\begin{equation}\label{f21-first}
q_{i}(x) = d_{i}(x) - 2\vartheta d(x)d_{i}(x)
\end{equation}
and
\begin{equation}\label{f22-second}
q_{ij}(x) = d_{ij}(x) - 2\vartheta d_{i}(x)d_{j}(x) - 2\vartheta
d(x)d_{ij}(x).
\end{equation}
For any $x_{0}\in\partial M^{n}$, after a rotation of the first
$n-1$ coordinates and remembering that $\mu(x_{0}) = e_{n}$, we have
\begin{equation*}
d_{ij}(x_{0}) =
\begin{pmatrix}
-\kappa_{1} & 0 & \cdots & 0 \\
\cdots & \cdots & \cdots & \cdots \\
0 & 0 & \cdots & -\kappa_{n-1} \\
0 & 0 & \cdots & 0
\end{pmatrix},
\end{equation*}
where there exists a constant $\kappa_{0} = \kappa_{0}(\partial
M^{n})
> 0$ such that $\kappa_{i} \geq \kappa_{0}$ for all principal
curvatures $\kappa_{i}$ of $\partial M^{n}$, $i = 1, 2, \cdots,
n-1$, and for any $x_{0}\in\partial M^{n}$. Since the differential
of the distance coincides with the past-directed spacelike normal
vector $-Dd(x_{0}) = \mu(x_{0}) = e_{n}$. Thus, it holds at $x_{0}$
\begin{equation*}
q_{ij}(x_{0}) =
\begin{pmatrix}
-\kappa_{1}(1-2\vartheta d) & 0 & \cdots & 0 \\
\cdots & \cdots & \cdots & \cdots \\
0 & 0 & \cdots & -\kappa_{n-1}(1-2\vartheta d) \\
0 & 0 & \cdots & -2\vartheta
\end{pmatrix}.
\end{equation*}
Choosing $\vartheta\delta\leq \frac{1}{4}$, we have
$$\iota^{ij}q_{ij} \leq -\frac{1}{2}\kappa_{0}
\left(\iota^{11}  + \iota^{22} + \cdots + \iota^{(n-1)(n-1)}\right)
- 2\vartheta\iota^{nn}.$$
 On the one hand, we can choose $\vartheta
\geq \frac{1}{4}\kappa_{0}$ such that
\begin{equation}\label{f23}
\iota^{ij}q_{ij} \leq -\frac{1}{2}\kappa_{0}\tr(\iota^{ij}).
\end{equation}
On the other hand, using the inequality of arithmetric and geometric
mean, we obtain
\begin{equation*}
\iota^{ij}q_{ij} \leq -C(n, \kappa_{0})\vartheta^{\frac{1}{n}}\left(
\prod_{i=1}^{n}\iota^{ii} \right)^{\frac{1}{n}}.
\end{equation*}
The Hadamand inequality for positive definite matrices
\begin{equation*}
\det(\iota^{ij}) \leq \left( \prod_{i=1}^{n}\iota^{ii} \right),
\end{equation*}
which implies
\begin{equation*}
\iota^{ij}q_{ij} \leq -C(n, \kappa_{0})
\vartheta^{\frac{1}{n}}{\det}^{\frac{1}{n}}(\iota^{ij}).
\end{equation*}
Recalling \eqref{det}, there is a positive constant $c_{4}$ such
that
\begin{equation*}
\det(\iota^{ij}) =\left(\det(\iota_{ij})\right)^{-1} \geq
\frac{1}{c_{4}}
>0,
\end{equation*}
and then it follows that
\begin{equation*}
\iota^{ij}q_{ij} \leq -\frac{1}{c_{4}}C(n,
\kappa_{0})\vartheta^{\frac{1}{n}}.
\end{equation*}
Using the $C^{1}$-estimate \eqref{Gra-est}, we have
\begin{equation*}
\left| \frac{2(n+1)}{n}\frac{1}{1-|D\varphi|^{2}}\varphi^{k}q_{k}
\right| = \left|
\frac{2(n+1)}{n}\frac{1}{1-|D\varphi|^{2}}\varphi^{k}(d_{k} -
2\vartheta d d_{k}) \right| \leq c_{5}(n, M_{0}^{n}, \rho)(1 +
\vartheta\delta),
\end{equation*}
for all $(x, t)\in M_{n}^{\delta}\times [0,T]$. Choose $\vartheta$
so large and $\delta$ so small such that
\begin{equation*}
\frac{1}{2n}\frac{1}{c_{4}}C(n, \kappa_{0})\vartheta^{\frac{1}{n}}
\geq c_{5}(n, M_{0}^{n}, \rho)(1 + \vartheta\delta).
\end{equation*}
Thus,  from \eqref{f23}, we have
\begin{equation*}
\begin{aligned}
\widetilde{\mathcal{L}}q(x) &=
\frac{1}{2n}\dot{\varphi}\iota^{ij}q_{ij} + \frac{2(n+1)}{n}\frac{\dot{\varphi}}{1-|D\varphi|^{2}}\varphi^{k}q_{k} + \frac{1}{2n}\dot{\varphi}\iota^{ij}q_{ij}\\
&\geq
-\frac{1}{2n}\dot{\varphi}\frac{1}{c_{4}}C(n, \kappa_{0})\vartheta^{\frac{1}{n}} + \dot{\varphi}c_{5}(n, M_{0}^{n}, \rho)(1 + \vartheta\delta) + \frac{1}{2n}\dot{\varphi}\iota^{ij}q_{ij}\\
&\geq -\frac{1}{4n}\kappa_{0}\dot{\varphi}\tr(\iota^{ij}).
\end{aligned}
\end{equation*}
 which finishes the proof of Lemma \ref{lemma 4.5}.
\end{proof}

Clearly, choosing $\frac{1}{8}\leq\vartheta\delta\leq\frac{1}{4}$,
from \eqref{f21-first} and \eqref{f22-second}, we can make sure that
$q$ satisfies the following properties in $M_{n}^{\delta}$:
\begin{equation*}
0 \leq q(x) \leq \delta,
\end{equation*}
\begin{equation}\label{f24}
\frac{1}{2} \leq |Dq| \leq 1,
\end{equation}
\begin{equation}\label{f25}
-\frac{\kappa_{0}}{2}\sigma_{ij} \geq D^{2}q \geq -C(\partial
M^{n})(1 + \vartheta)\sigma_{ij},
\end{equation}
and
\begin{equation}\label{f26}
|D^{3}q| \leq C(\partial M^{n})(1 + \vartheta).
\end{equation}
It's easy to see
\begin{equation*}\label{f27}
\frac{Dq}{|Dq|} = -\mu
\end{equation*}
for  future-directed spacelike unit normal vector $\mu$ on the
boundary $\partial M^{n}$. Consider the following function
\begin{equation*}
P(x,t) = \varphi_{i}q^{i} - Bq(x),
\end{equation*}
where the constant $B$ will be chosen later.

\begin{lemma}\label{lemma 4.6}
For any solution $\varphi$ of the flow \eqref{Evo-1} in
$M_{n}\times[0,T]$ for some fixed $T < T^{*}$, we have
\begin{equation*}
\widetilde{\mathcal{L}}P(x,t) \leq 0.
\end{equation*}
\end{lemma}

\begin{proof}
The calculation of $\widetilde{\mathcal{L}}P(x,t)$ is similar to
that of \eqref{evolu 3}. Differentiating this function $P(x,t)$
w.r.t. $x$ twice yields
\begin{equation*}
P_{i} = \varphi_{li}q^{l} + \varphi_{l}q^{l}_{,i} - Bq_{i}
\end{equation*}
and
\begin{equation*}
P_{ij} = \varphi_{lij}q^{l} + \varphi_{li}q^{l}_{,j} +
\varphi_{lj}q^{l}_{,i} + \varphi^{l}q_{lij} - Bq_{ij}.
\end{equation*}
Differentiating $P(x,t)$ w.r.t. $t$, we have
\begin{equation*}
\begin{aligned}
P_{t} &= \varphi_{lt}q^{l}\\
&=
\left( Q^{ij}\varphi_{ijl} + Q^{k}\varphi_{kl} \right) q^{l}\\
&= \left( Q^{ij}\varphi_{lij} + Q^{ij}\varphi_{l}\sigma_{ij}
 - Q^{ij}\varphi_{j}\sigma_{il} + Q^{k}\varphi_{kl} \right)q^{l}.
\end{aligned}
\end{equation*}
Therefore, we have
\begin{equation*}
\begin{aligned}
\widetilde{\mathcal{L}}P(x,t) &= P_{t} - Q^{ij}P_{ij} +
\frac{2(n+1)}{n}\frac{\dot{\varphi}}{1-|D\varphi|^{2}}
\varphi^{k}P_{k}\\
&=
- 2Q^{ij}\varphi_{li}q^{l}_{,j} + Q^{ij}(\sigma_{ij}\varphi_{l}q^{l} - \sigma_{il}\varphi_{j}q^{l}) - Q^{ij}\varphi^{l}q_{lij}\\
&\quad +
\frac{2(n+1)}{n}\frac{\dot{\varphi}}{1-|D\varphi|^{2}}\varphi^{k}\varphi^{l}q_{lk}
+ \frac{2\dot{\varphi}}{n}\iota^{kl}\varphi_{l}\varphi_{km}q^{m} -
  B\widetilde{\mathcal{L}}q(x).
\end{aligned}
\end{equation*}
Since
\begin{equation*}
\iota^{kl}\varphi_{km} = \delta^{l}_{m} - \iota^{kl}\sigma_{km} +
\iota^{kl}\varphi_{k}\varphi_{m},
\end{equation*}
by using \eqref{f24}, \eqref{f25}, \eqref{f26} and the
$C^{1}$-estimate \eqref{Gra-est}, we get
\begin{equation*}
\widetilde{\mathcal{L}}P(x,t) \leq -C(n, M_{0}^{n}, \rho,
\kappa_{0})(1+\vartheta)\dot{\varphi}\tr\iota^{ij}
 - C(n, M_{0}^{n}, \rho, \kappa_{0})(1+\vartheta)\dot{\varphi} - B\widetilde{\mathcal{L}}q(x).
\end{equation*}
Using Lemma \ref{lemma 4.5}, we have
\begin{equation*}
\widetilde{\mathcal{L}}P(x,t) \leq -C(n, M_{0}^{n}, \rho,
\kappa_{0})\dot{\varphi}\left( (1 + \vartheta - B)\tr\iota^{ij} + (1
+ \vartheta) \right).
\end{equation*}
Recalling \eqref{det}, it follows that
\begin{equation*}
\left(\frac{\tr\iota^{ij}}{n}\right)^{n} \geq \det(\iota^{ij}) =
\left(\det(\iota_{ij})\right)^{-1} \geq \frac{1}{c_{4}} > 0.
\end{equation*}
Since $\dot{\varphi} < 0$, choosing $B \geq
\frac{c_{4}^{\frac{1}{n}}}{n}(1 + \vartheta) + 1 + \vartheta$, we
get
\begin{equation*}
\widetilde{\mathcal{L}}P(x,t) \leq 0,
\end{equation*}
which completes the proof.
\end{proof}

We can obtain:

\begin{proposition}\label{proposition 4.7}
Let $\varphi$ be a solution of the flow \eqref{Evo-1} in
$M^{n}\times[0,T]$ for some fixed $T<T^{*}$. Then $\varphi_{\mu\mu}$
is uniformly bounded from above, i.e., there exists $C=C(n,
M_{0}^{n})$ such that
\begin{equation*}
\varphi_{\mu\mu} \leq C(n, M_{0}^{n}) \qquad
\forall~(x,t)\in\partial M^{n}\times[0,T],
\end{equation*}
where $\varphi_{\mu\mu} := \varphi_{ij}\mu^{i}\mu^{j}.$
\end{proposition}

\begin{proof}
From the boundary condition in \eqref{Evo-1}, it is easy to know
\begin{equation*}
P=0 \qquad \mathrm{on}~~\partial M^{n}\times[0,T].
\end{equation*}
On $(\partial M_{n}^{\delta} \backslash \partial M^{n}
)\times[0,T]$, we have
\begin{equation*}
P \leq C(\rho) - \frac{3}{4}B\varepsilon \leq 0
\end{equation*}
provided $B \geq \frac{4}{3}\frac{C(\rho)}{\varepsilon}$, where
$\varepsilon\in(0, \delta]$. Applying the maximum principle, it
follows that
\begin{equation*}
P \leq 0    \qquad\qquad \mathrm{in}~~M_{n}^{\delta}\times[0,T].
\end{equation*}
We have further that $P(x_{0}, t) = 0$ is a maximum, which implies
\begin{equation*}
P_{i}\mu^{i} \geq 0
\end{equation*}
for the future-directed spacelike unit normal vector $\mu$.
Therefore,
\begin{equation*}
-\varphi_{\mu\mu}|Dq| + \varphi_{j}q^{j}_{,i}\mu^{i} - Bq_{i}\mu^{i}
\geq 0.
\end{equation*}
Finally, using the $C^{1}$-estimate \eqref{Gra-est}, we have
\begin{equation*}
\varphi_{\mu\mu} \leq C(n, M_{0}^{n}),
\end{equation*}
which completes the proof.
\end{proof}

\subsection{Remaining $C^{2}$ boundary estimates.}

We have obtained interior estimates under the assumption that the
maximum of $S$ is attained in the interior of $M^{n}$. Now, we have
to consider the possibility that the maximum of $S$ is not in the
interior of $M^{n}$. Since the double normal boundary estimates have
been done in the previous subsection, we shall try to get remaining
$C^{2}$ boundary estimates by following a similar discussion to that
done by Lions-Trudinger-Urbas in \cite{PL}.

\begin{proposition}\label{proposition 4.8}
Let $\varphi$ be a solution of the flow \eqref{Evo-1} in
$M^{n}\times[0,T]$ for some fixed $T < T^{\ast}$, and assume that
$S$ attains its maximum on $\partial
M^{n}\times\mathscr{H}^{n}(1)\times[0,T]$. Then, there exists $C =
C(n, M_{0}^{n})$ such that
\begin{equation*}
\varphi_{ij}(x,t)\xi^{i}\xi^{j} \leq C(n, M_{0}^{n}) \qquad
\forall~(x, \xi, t)\in \partial
M^{n}\times\mathscr{H}^{n}(1)\times[0,T].
\end{equation*}
\end{proposition}

\begin{proof}
Assume that $S$ attains its maximum at a point $(x_{0}, \xi_{0},
t_{0})\in \partial M^{n}\times\mathscr{H}^{n}(1)\times[0,T]$. By
proposition \ref{proposition 4.7}, we know
\begin{equation*}
\varphi_{\mu\mu} \leq C(n, M_{0}^{n}) \qquad \forall~(x,t)\in
\partial M^{n}\times[0, T].
\end{equation*}
Thus, the remaining case is $\xi_{0} \neq \mu$. Without loss of
generality, assume that $S$ attains its maximum at a point $(x_{0},
\xi_{0}, t_{0})\in \partial
M^{n}\times\mathscr{H}^{n}(1)\times[0,T]$ with $\xi_{0} \neq \mu$.
Let $x_{0}\in\partial M^{n}$ be fixed. We choose a boundary
coordinate chart containing $x_{0}$ so that $\partial M^{n}$ is
represented locally as a graph of function $f$ over its tangent
plane at $x_{0} = (\widehat{x}_{0}, x_{0}^{n})$, and then locally
$M^{n} = \{(\widehat{x}, x^{n}) | x^{n} < f(\widehat{x})\}$,
$Df(\widehat{x}_{0}) = 0$. We divide the rest discussion into two
cases:

\textbf{Case 1}. Assume that $\xi_{0}$ is tangential. If $\xi_{0}$
is tangential to $\partial M^{n}$, differentiating the boundary
condition
$$\mu^{i}\varphi_{i} = 0$$
w.r.t. tangential directions $\xi_{0}$ yields
$$\mu^{i}_{,\xi_{0}}\varphi_{i} + \mu^{i}\varphi_{i\xi_{0}} + \mu^{i}\varphi_{in}f_{\xi_{0}} = 0,$$
 and then at $x_{0}$, it follows that
$$\mu^{i}_{,\xi_{0}}\varphi_{i} + \mu^{i}\varphi_{i\xi_{0}} = 0$$
in view of $Df(\widehat{x}_{0}) = 0$. This, together with the
$C^{1}$-estimate \eqref{Gra-est},  implies
$$|\mu^{i}\varphi_{i\xi_{0}}| \leq C(n, M_{0}^{n}).$$
Differentiating the boundary condition again, we can get at $x_{0}$,
$$\mu^{i}_{,\xi_{0}\xi_{0}}\varphi_{i} + 2\mu^{i}_{,\xi_{0}}\varphi_{i\xi_{0}} + \mu^{i}\varphi_{i\xi_{0}\xi_{0}} + \mu^{i}\varphi_{in}f_{\xi_{0}\xi_{0}} = 0$$
in view of $Df(\widehat{x}_{0}) = 0$.

In the rest part of the proof, we make the following agreement:
 \begin{itemize}
 \item We put vectors as indices to indicate products as
$$\varphi_{\mu\xi\xi} := \mu^{i}\varphi_{ijk}\xi^{j}\xi^{k}$$
but not covariant derivatives in the corresponding direction
$$\varphi_{i\xi\xi} \neq (\xi^{j}\varphi_{ij})_{,\xi} = \xi^{k}\xi^{j}_{,k}\varphi_{ij} + \xi^{k}\xi^{j}\varphi_{ijk}.$$
Analogously, we have
$$\varphi_{\xi\xi} := \varphi_{ij}\xi^{i}\xi^{j}, \quad \iota_{\xi\xi} :=\iota_{ij}\xi^{i}\xi^{j}, \quad \iota_{\mu\mu} := \iota_{ij}\mu^{i}\mu^{j}, \quad \sigma_{\xi\xi} := \sigma_{ij}\xi^{i}\xi^{j}, \quad \cdots$$
 \end{itemize}
At $x_{0}$, $C^{1}$-estimate \eqref{Gra-est} and double normal
estimates provide
$$\mu^{i}_{,\xi_{0}\xi_{0}}\varphi_{i} \geq -C(\partial M^{n}, \rho)$$
and
$$\mu^{i}\varphi_{in}f_{\xi_{0}\xi_{0}} \geq -C(n, M_{0}^{n}, \partial M^{n})$$
in view of $D^{2}f(\widehat{x}_{0}) < 0$. So, we have
\begin{equation}\label{f28}
\begin{aligned}
\varphi_{\mu\xi_{0}\xi_{0}} &\leq
-2\mu^{i}_{,\xi_{0}}\varphi_{i\xi_{0}} + C(n, M_{0}^{n}, \rho, \partial M^{n})\\
&\leq
-2\mu^{i}_{,\xi_{0}}(1 + \varphi_{i\xi_{0}} - \varphi_{i}\varphi_{\xi_{0}}) + C(n, M_{0}^{n}, \partial M^{n}, \rho)\\
&= -2\mu^{i}_{,\xi_{0}}\iota_{i\xi_{0}} + C(n, M_{0}^{n}, \partial
M^{n}, \rho).
\end{aligned}
\end{equation}
As already noted, $\xi_{0}$ is an eigenvector of $\iota_{ij}(x_{0},
t_{0}) + \eta_{ij}(x_{0}, t_{0})$ with an eigenvalue $\lambda_{0}$,
since it corresponds to a maximal direction. Therefore, it follows
that
\begin{equation*}
\begin{aligned}
-\mu^{i}_{,\xi_{0}}\iota_{i\xi_{0}}(x_{0}, t_{0}) &=
-\xi^{j}_{0}\mu^{i}_{,j}(\iota_{ik} + \eta_{ik})\xi^{k}_{0} + \xi^{j}_{0}\mu^{i}_{,j}\eta_{ik}\xi^{k}_{0}\\
&= -\lambda_{0}\xi^{j}_{0}\mu^{i}_{,j}\sigma_{ik}\xi^{k}_{0} +
\xi^{j}_{0}\mu^{i}_{,j}\eta_{ik}\xi^{k}_{0},
\end{aligned}
\end{equation*}
where we may assume that $\lambda_{0}$ is nonnegative, because
otherwise $\iota_{ik} + \eta_{ik}$ would be negative definite and
the needed estimate would follow immediately. Moreover, the strict
convexity of $\partial M^{n}$ implies the existence of a constant
$C>0$ such that
$$\xi^{j}\mu^{i}_{,j}\sigma_{ik}\xi^{k} \geq C(\partial M^{n})\xi^{i}\xi^{k}\sigma_{ik}$$
for all tangential vectors $\xi$. Thus, the inequality \eqref{f28}
degenerates into
\begin{equation*}
\begin{aligned}
\varphi_{\mu\xi_{0}\xi_{0}} &\leq -2\lambda_{0}\xi^{j}_{0}\mu^{i}_{,j}\sigma_{ik}\xi^{k}_{0} + 2\xi^{j}_{0}\mu^{i}_{,j}\eta_{ik}\xi^{k}_{0} + C(n, M_{0}^{n}, \partial M^{n}, \rho)\\
&=
-2\xi^{j}_{0}(\iota_{ik} + \eta_{ik})\mu^{i}_{,j}\xi^{k}_{0} + 2\xi^{j}_{0}\mu^{i}_{,j}\eta_{ik}\xi^{k}_{0} + C(n, M_{0}^{n}, \partial M^{n}, \rho)\\
&\leq -2C(\partial M^{n})\iota_{\xi_{0}\xi_{0}} + C(n, M_{0}^{n},
\partial M^{n}, \rho).
\end{aligned}
\end{equation*}
On the other hand, since $S$ achieves its maximal value at $x_{0}$,
it gives $0 \leq S_{\mu}$,
$$0 \leq \frac{\iota_{\xi_{0}\xi_{0},\mu} + \eta_{\xi_{0}\xi_{0}, \mu}}{V_{\xi_{0}}} + \lambda\varphi^{i}\varphi_{i\mu},$$
where $V_{\xi_{0}} := \iota_{\xi_{0}\xi_{0}} + \eta_{\xi_{0}\xi_{0}}
+ C$. Since $\partial M^{n}$ is strictly convex, we have
$$\lambda\varphi^{i}\varphi_{i\mu} = -\lambda\varphi^{i}\mu^{j}_{,i}\varphi_{j} \leq -C(\partial M^{n})\lambda|D\varphi|^{2} \leq 0,$$
which implies
$$0 \leq \varphi_{\xi_{0}\xi_{0}\mu} + C(n, M_{0}^{n}, \partial M^{n}).$$
Together with
$$\varphi_{\xi_{0}\xi_{0}\mu} = \varphi_{\mu\xi_{0}\xi_{0}} + R_{i\xi_{0}\xi_{0}\mu}\varphi^{i},$$
 we can get
$$S(x_{0}, \xi_{0}, t_{0}) \leq C(n, M_{0}^{n}).$$
So, the desired estimate
$$\varphi_{ij}(x,t)\xi^{i}\xi^{j} \leq C(n, M_{0}^{n}), \qquad \forall~(x, \xi, t)\in \overline{M^{n}}\times\mathscr{H}^{n}(1)\times[0,T]$$
follows in \textbf{Case 1}.

\textbf{Case 2}. Assume that $\xi_{0}$ is non-tangential. If
$\xi_{0}$ is neither tangential nor normal, we use the tricky choice
introduced in \cite{PL}. We find $0<\ell<1$ and a tangential
direction $\gamma$ such that
$$\xi_{0} = \ell\gamma + \sqrt{1-\ell^{2}}\mu.$$
Thus,
\begin{equation*}
\begin{aligned}
\varphi_{\xi_{0}\xi_{0}}&=
\varphi_{ij}\xi^{i}_{0}\xi^{j}_{0} = \varphi_{ij}\left(\ell\gamma^{i} + \sqrt{1-\ell^{2}} \mu^{i}\right)\left(\ell\gamma^{j} + \sqrt{1-\ell^{2}}\mu^{j} \right)\\
&=
\varphi_{ij}\left(\ell^{2}\gamma^{i}\gamma^{j} + 2p\sqrt{1-\ell^{2}}\gamma^{i}\mu^{j} + (1-\ell^{2})\mu^{i}\mu^{j} \right)\\
&=\ell^{2}\varphi_{\gamma\gamma} +
2\ell\sqrt{1-\ell^{2}}\varphi_{\gamma\mu} +
(1-\ell^{2})\varphi_{\mu\mu}.
\end{aligned}
\end{equation*}
Differentiating the boundary condition at a boundary point, we have
$$\mu^{i}_{,j}\varphi_{i} = -\mu^{i}\varphi_{ij}.$$
Therefore, at the boundary point, one has
$$\eta(x, \xi_{0}, \xi_{0}, t) = 2\mu^{i}_{,j}\varphi_{i}\langle\xi_{0},
\mu\rangle = -2\ell\sqrt{1-\ell^{2}}\varphi_{\gamma\mu},$$ and
consequently,
$$\varphi_{\xi_{0}\xi_{0}} = \ell^{2}\varphi_{\gamma\gamma} + (1-\ell^{2})\varphi_{\mu\mu} - \eta_{\xi_{0}\xi_{0}}.$$
Thus, in view of the NBC in \eqref{Evo-1}, one has
\begin{equation*}
\begin{aligned}
\iota_{\xi_{0}\xi_{0}} + \eta_{\xi_{0}\xi_{0}} &= 1-\ell^{2}\varphi_{\gamma}\varphi_{\gamma} + \left(\ell^{2}\varphi_{\gamma\gamma} + (1-\ell^{2})\varphi_{\mu\mu}\right)\\
&= \ell^{2}\iota_{\gamma\gamma} + (1-\ell^{2})\iota_{\mu\mu}.
\end{aligned}
\end{equation*}
Therefore, the identity $\eta(x, \gamma, \gamma, t) = \eta(x, \mu,
\mu, t) = 0$ for all $(x,t)\in\overline{M^{n}}\times[0,T]$ provides
\begin{equation*}
\begin{aligned}
\iota_{\xi_{0}\xi_{0}} + \eta_{\xi_{0}\xi_{0}} &=
\ell^{2}(\iota_{\gamma\gamma }+ \eta_{\gamma\gamma}) + (1-\ell^{2})(\iota_{\mu\mu} + \eta_{\mu\mu})\\
&\leq \ell^{2}(\iota_{\xi_{0}\xi_{0}} + \eta_{\xi_{0}\xi_{0}}) +
(1-\ell^{2})(\iota_{\mu\mu} + \eta_{\mu\mu}).
\end{aligned}
\end{equation*}
It follows immediately
$$\iota_{\xi_{0}\xi_{0}} + \eta_{\xi_{0}\xi_{0}} \leq \iota_{\mu\mu} + \eta_{\mu\mu}.$$
 Then we have
$$\iota_{\xi_{0}\xi_{0}}(x_{0}, t_{0}) \leq \iota_{\mu\mu}(x_{0}, t_{0}) - \eta_{\xi_{0}\xi_{0}} \leq \iota_{\mu\mu}(x_{0}, t_{0}) + C(M_{0}^{n}).$$
This implies
$$\iota_{\xi_{0}\xi_{0}} \leq C(n, M_{0}^{n})$$
in view of Proposition \ref{proposition 4.7}, and the desired
estimate
$$\varphi_{ij}\xi^{i}\xi^{j} \leq C(n, M_{0}^{n})   \qquad \forall~(x, \xi, t)\in\partial M^{n}\times\mathscr{H}^{n}(1)\times[0,T]$$
follows in \textbf{Case 2}. Our proof is finished.
\end{proof}

\begin{theorem}\label{theorem 4.9}
Under the hypothesis of Theorem \ref{main1.1}, we conclude
$$T^{*} = + \infty.$$
\end{theorem}

\begin{proof}
Recalling that $\varphi$ satisfies the system \eqref{Evo-1},
$$\frac{\partial\varphi}{\partial t} = Q(D\varphi, D^{2}\varphi).$$
By a simple calculation, we get
$$\frac{\partial Q}{\partial\varphi_{ij}} = \frac{1}{n}(1-|D\varphi|^{2})^{\frac{n+1}{n}}\frac{\det^{\frac{1}{n}}(\sigma_{ij})}{\det^{\frac{1}{n}}(\iota_{ij})}\iota^{ij},$$
which is uniformly parabolic on finite intervals from
$C^{0}$-estimate, $C^{1}$-estimate \eqref{Gra-est} and the estimate
\eqref{det}. Then by the Krylov-Safanov estimates\footnote{~Using
Harnack inequality and Alexandrov's maximum principle, which yields
an estimate on the $L^{\infty}$-norms of solutions in terms of
$L^{n}$-norms of nonhomogeneous terms, the related H\"{o}lder
estimate can be obtained.} (see, e.g., \cite[Chapter 14]{GM}), we
have
\begin{equation*}
\Vert\varphi \Vert_{C^{2, \alpha}(\overline{M^{n}})} \leq C(n,
M_{0}^{n})
\end{equation*}
which implies the maximal time interval is unbounded, i.e.,
$T^{\ast} = +\infty$.
\end{proof}

Taking the derivatives to the evolution equation in \eqref{Evo-1},
and then using a similar argument as above (i.e., using the method
of Krylov-Safanov estimates), we can increase the regularity of the
solution of the flow from $C^{2,\alpha}$ to $C^{3,\alpha}$, and so
on, the smoothness of the solution to the flow can be obtained.

\section{Convergence of the rescaled flow} \label{se6}

Now, we define the rescaled flow by
\begin{equation*}
\widetilde{X} = X\Lambda^{-1},
\end{equation*}
where $\Lambda:=e^{-t}$. Thus,
\begin{equation*}
\widetilde{u} = u\Lambda^{-1},
\end{equation*}
\begin{equation*}
\widetilde{\varphi} = \varphi - \log\Lambda,
\end{equation*}
and the rescaled Gaussian curvature is
\begin{equation*}
\widetilde{K} = K\Lambda^{n}
\end{equation*}
Then, the rescaled scalar curvature equation takes the form
\begin{equation*}
\frac{\partial}{\partial t}\widetilde{u} = -
v\widetilde{K}^{-\frac{1}{n}} + \widetilde{u},
\end{equation*}
or equivalently, with $\widetilde{\varphi} = \log\widetilde{u}$,
\begin{equation}\label{rescale 1}
\frac{\partial}{\partial t}\widetilde{\varphi} =
-v\widetilde{u}^{-1}\widetilde{K}^{-\frac{1}{n}} + 1 =
\widetilde{Q}(D\widetilde{\varphi}, D^{2}\widetilde{\varphi}).
\end{equation}
Since the spatial derivatives of $\widetilde{\varphi}$ are equal to
those of $\varphi$, \eqref{rescale 1} is a second-order nonlinear
parabolic PDE with a uniformly parabolic and concave operator
$\widetilde{K}$. The rescaled version of the system \eqref{Evo-1}
satisfies
\begin{equation*}
\left\{
\begin{aligned}
&\frac{\partial}{\partial
t}\widetilde{\varphi}=\widetilde{Q}(D\widetilde{\varphi},
D^2\widetilde{\varphi})  \qquad &&\mathrm{in}~
M^n\times(0,T)\\
&\nabla_{\mu} \widetilde{\varphi}=0  \qquad &&\mathrm{on}~ \partial M^n\times(0,T)\\
&\widetilde{\varphi}(\cdot,0)=\widetilde{\varphi}_{0} \qquad
&&\mathrm{in}~M^n,
\end{aligned}
\right.
\end{equation*}
where
$$\widetilde{Q}(D\widetilde{\varphi}, D^{2}\widetilde{\varphi}) :=
-(1-|D\widetilde{\varphi}|^{2})^{\frac{n+1}{n}}\frac{{\det}^{\frac{1}{n}}(\sigma_{ij})}{{\det}^{\frac{1}{n}}(\iota_{ij})}
+ 1.$$
  Then, using the decay estimate \eqref{Gra-est} of
$|D\varphi|$, we can deduce a decay estimate of
$|D\widetilde{\varphi}(\cdot, t)|$ as follows:

\begin{lemma}\label{lemma 5.1}
Let $\varphi$ be a solution of \eqref{Evo-1}, then we have
\begin{equation*}\label{rescale 2}
|D\widetilde{\varphi}(x,t)| \leq
\sup\limits_{\overline{M^{n}}}|D\widetilde{\varphi}(\cdot, 0)| \leq
\rho < 1.
\end{equation*}
\end{lemma}

\begin{proof}
Set $\widetilde{\Phi} = \frac{|D\widetilde{\varphi}|^{2}}{2}$.
Similar to the argument in Lemma \ref{Gradient}, we can obtain
\begin{equation*}
\frac{\partial\widetilde{\Phi}}{\partial t} =
\widetilde{Q}^{ij}\widetilde{\Phi}_{ij} +
\widetilde{Q}^{k}\widetilde{\Phi}_{k} - \widetilde{Q}^{ij}\left(
\widetilde{\varphi}_{i}\widetilde{\varphi}_{j} -
\sigma_{ij}|D\widetilde{\varphi}|^{2} \right) -
\widetilde{Q}^{ij}\widetilde{\varphi}_{mi}\widetilde{\varphi}^{m}_{j},
\end{equation*}
with the boundary condition
$$\nabla_{\mu} \widetilde{\Phi} \leq 0.$$
So, we have
\begin{equation*}
\left\{
\begin{aligned}
&\frac{\partial}{\partial t}\widetilde{\Phi} \leq
\widetilde{Q}^{ij}\widetilde{\Phi}_{ij} +
\widetilde{Q}^{k}\widetilde{\Phi}_{k}  \qquad &&\mathrm{in}~
M^n\times(0,\infty)\\
&\nabla_{\mu} \widetilde{\Phi} \leq 0  \qquad &&\mathrm{on}~ \partial M^n\times(0,\infty)\\
&\widetilde{\Phi}(\cdot,0)= \frac{|D\widetilde{\varphi}(\cdot,
0)|^{2}}{2} \qquad &&\mathrm{in}~M^n.
\end{aligned}
\right.
\end{equation*}
Using the maximum principle and Hopf's Lemma, we can get the
gradient estimates of $\widetilde{\varphi}$.
\end{proof}

\begin{lemma}\label{lemma 5.2}
Let $\varphi$ be a solution of the inverse Gauss curvature flow
\eqref{Evo-1}. Then,
\begin{equation*}
\widetilde{\varphi}(\cdot, t) = \varphi(\cdot, t) + t
\end{equation*}
converges to a real number for $t\to +\infty$.
\end{lemma}

\begin{proof}
Using Lemma \ref{$C^{0}$ estimate}, the Evans-Krylov theorem (see
\cite{Eck,LCEvans}) and thereafter the parabolic Schauder estimate,
we can prove this lemma.
\end{proof}

So, we have:

\begin{theorem}\label{theorem 5.3}
The rescaled flow
\begin{equation*}
\frac{d\widetilde{X}}{dt} = \widetilde{K}^{-\frac{1}{n}}\nu +
\widetilde{X}
\end{equation*}
exists for all time and the leaves converge in $C^{\infty}$ to a
piece of the spacelike graph of some positive constant function
defined over $M^{n}$, i.e. a piece of hyperbolic plane of center at
origin and prescribed radius.
\end{theorem}

\vspace{5mm}

\section*{Acknowledgments}
This work is partially supported by the NSF of China (Grant Nos.
11801496 and 11926352), the Fok Ying-Tung Education Foundation
(China) and  Hubei Key Laboratory of Applied Mathematics (Hubei
University).

\end{document}